\documentclass[11pt]{amsart}

\usepackage{amscd,amssymb,amsopn,amsmath,amsthm,mathrsfs,graphics,amsfonts,enumerate,verbatim,calc
}

\usepackage[all]{xy}
\usepackage[normalem]{ulem}
\usepackage[colorlinks=true,linkcolor=blue,citecolor=blue]{hyperref}

\usepackage{color}

\usepackage[OT2,OT1]{fontenc}
\newcommand\cyr{%
\renewcommand\rmdefault{wncyr}%
\renewcommand\sfdefault{wncyss}%
\renewcommand\encodingdefault{OT2}%
\normalfont
\selectfont}
\DeclareTextFontCommand{\textcyr}{\cyr}

\DeclareFontFamily{OT1}{rsfs}{}
\DeclareFontShape{OT1}{rsfs}{n}{it}{<-> rsfs10}{}
\DeclareMathAlphabet{\mathscr}{OT1}{rsfs}{n}{it}

\numberwithin{equation}{section}
\newtheorem{theorem}{Theorem}[section]
\newtheorem{lemma}[theorem]{Lemma}
\newtheorem{proposition}[theorem]{Proposition}
\newtheorem{corollary}[theorem]{Corollary}

\newtheorem{question}{Question}

\theoremstyle{definition}
\newtheorem{definition}[theorem]{Definition}

\theoremstyle{remark}

\newtheorem{example}[theorem]{Example}

\newcommand{\Ass}{\operatorname{Ass}}
\newcommand{\Assh}{\operatorname{Assh}}

\newcommand{\Hom}{\operatorname{Hom}}

\newcommand{\Ann}{\operatorname{Ann}}

\newcommand{\depth}{\operatorname{depth}}

\newcommand{\syz}{\operatorname{syz}}

\newcommand{\N}{\mathbb N}
\newcommand{\Z}{\mathbb Z}

\newcommand{\noi}{\mbox{i-effaceable}}

\begin{document}
\title[Finite length local cohomology and bounds on Koszul cohomology]{Characterizing finite length local cohomology in terms of bounds on Koszul cohomology}

\author[Patricia Klein]{Patricia Klein}
\address{Department of Mathematics, University of Kentucky, Lexington, KY 40506 USA}
\email{pklein@uky.edu}

\maketitle

\section{Abstract}
Let $(R,m, \kappa)$ be a local ring.  We give a characterization of $R$-modules $M$ whose local cohomology is finite length up to some index in terms of asymptotic vanishing of Koszul cohomology on parameter ideals up to the same index.  In particular, we show that a quasi-unmixed module $M$ is asymptotically Cohen--Macaulay if and only if $M$ is  Cohen--Macaulay on the punctured spectrum if and only if $\sup\{\ell(H^i(f_1, \ldots, f_d;M))\mid \sqrt{f_1, \ldots, f_d} = m \mbox{, } i< d\}<\infty$ for $d = \dim(M) = \dim(R)$.

\section{Introduction} 

Throughout this paper, all rings are assumed to be commutative rings with unity, and all modules are assumed to be unital.  We will use $(R, m, \kappa)$ to denote a local ring, by which we mean a Noetherian ring with unique maximal ideal $m$ and residue field $\kappa$.  The goal of this paper is to give precise conditions under which lengths of Koszul cohomology of a finitely-generated  $R$-module $M$ on all systems of parameters on $M$ are bounded.  These conditions will be whether or not certain local cohomology of $M$ on the maximal ideal have finite length.  Looking to bound lengths of Koszul cohomology on systems of parameters in this way is natural in light of theorems of Lech's, described below, together with their recent strengthings, and also timely in light of work on similar asymptotics, also described below.

Recently, there has been a great deal of interest in the asymptotics of Koszul cohomology.  Bhatt, Hochster, and Ma have recently shown that several results, including positivity of Serre intersection multiplicities, that would follow from the existence of small Cohen--Macaulay modules over complete local domains would also follow from the existence of \emph{lim Cohen--Macaulay sequences} of modules, i.e. sequences of modules whose Koszul cohomology on a fixed system of parameters is eventually ``small enough" in a sense we will not make precise here \cite{BHM18, MelExp}.  In this vein, one may wonder if there are other environments in which Koszul cohomology vanishing may be replaced as a condition by Koszul cohomology being asymptotically small.  Some theorems of Lech's, together with more recent work in their wake, point to Koszul cohomology we would reasonably expect to be asymptotically small in this way.

Lech's inequality states that for any local ring $(R,m, \kappa)$ of dimension $d$ and any $m$-primary ideal $I$, $\dfrac{e_I(R)}{\ell(R/I)} \leq d! \cdot e_m(R)$ \cite[Theorem 3]{Lech60}.  Lech's inequality is an invaluable tool in many areas of commutative algebra, such as in the study of the minimal number of generators of ideals \cite{Bor79}, of first Hilbert coefficients \cite{PUV05}, and of reduction numbers \cite{Vas03}.  There has additionally been recent work to improve Lech's inequality, most notably in \cite{Hanes} and \cite{HSV17}.  Most recently, Lech's inequality was generalized to the case of finitely generated modules in \cite{joint}.  In the same paper, it is shown that there also exists a nonzero lower bound depending only on $M$ for the ratio $\dfrac{e_I(M)}{\ell(M/IM)}$ for all $m$-primary ideals $I$ whenever $M$ is a finitely generated quasi-unmixed $R$-module, which is to say that the St\"{u}ckrad--Vogel Conjecture is settled in the affirmative.  Furthermore, the authors show that $\frac{\ell(H^i(x_1^t, \ldots, x_d^t; M))}{\ell(M/(x_1^t, \ldots, x_d^t)M)}$ converges to $0$ \emph{uniformly} over all systems of parameters $x_1, \ldots, x_d$ for every $0 \leq i < \dim(M)$ whenever $M$ is quasi-unmixed, where $H^i(x_1, \ldots, x_d; M)$ denotes the $i^{th}$ Koszul cohomology module of the elements $x_1, \ldots, x_d$ on the module $M$.  Returning to classical work due to Lech, the Lech's limit formula \cite[Theorem 2]{Lech60}, generalized can be generalized to modules (see \cite{Northcott}), states that \[
\displaystyle\lim_{min_i\{n_i \}\rightarrow \infty} \dfrac{e_{(x_1^{n_1}, \ldots, x_d^{n_d})}(M)}{\ell(M/(x_1^{n_1}, \ldots, x_d^{n_d})M)} = 1.
\]  

In light of these formulas, a natural question is whether Lech's limit formula holds when the sequence of parameter ideals given by powers of a fixed system of parameters is replaced by any sequence of parameter ideals in increasingly high powers of the maximal ideal.  Example \ref{badmultiplicity} demonstrates that it does not in general.  Moreover, Theorem \ref{ACMLCM} gives precise conditions on when $\displaystyle \lim_{n \to\infty} \dfrac{\ell(H^i(x_{1n}, \ldots, x_{dn},M))}{\ell(R/I_n)} = 0$ for all sequences of parameter ideals $(x_{1n}, \ldots, x_{dn}) = I_n \subseteq m^n$ and all $i <\dim(M) = \dim(R)$.  If $\Assh(M) = \Assh(R)$, these same conditions also give $\displaystyle \lim_{n \to\infty} \dfrac{\ell(H^i(x_{1n}, \ldots, x_{dn},M))}{\ell(M/I_nM)} = 0$ for all sequences of parameter ideals $(x_{1n}, \ldots, x_{dn}) = I_n \subseteq m^n$ and all $i <\dim(M)$, which directly implies that $\displaystyle \lim_{n \to\infty} \dfrac{e_{I_n}(M)}{\ell(M/I_nM)} = 1$ by Serre's formula \cite{Serre}, which states that \begin{equation}\label{SerreFormula}
e_I(R) = \displaystyle\sum_{i=0}^d (-1)^{i} \ell(H^{d-i}(x_1, \ldots, x_d;R)). 
\end{equation} 

For our convenience in studying these asymptotics, we make two definitions: 

\begin{definition} We say that $M$ is \emph{$\noi$} if for every sequence of parameter ideals $(x_{1n}, \ldots, x_{dn}) = I_n \subseteq m^n$, we have $\displaystyle \lim_{n \to\infty} \dfrac{\ell(H^i(x_{1n}, \ldots, x_{dn},M))}{\ell(R/I_nR)} = 0$.
\end{definition}

Define $\Assh(M) = \{P \in \Ass(M) \mid \dim(R/P) = \dim(M)\}$.  If $\Assh(M)  = \Assh(R)$, we may replace $\ell(R/I_nR)$ by $\ell(M/I_nM)$ in the definition above \cite[Lemma 4.4]{joint}.

\begin{definition} The \emph{asymptotic depth} of $M$, denoted asydepth $(M)$,  is $k$ if $M$ is $\noi$ for all $i<k$ and $M$ is not $k$-effaceable.  We say that $M$ is \emph{asymptotically Cohen--Macaulay} if asydepth $M = \dim M$.  
\end{definition}

We are now prepared to state the main theorem of this paper.

\begin{theorem}\label{main}
If $M$ is a finitely generated module over the local ring $(R,m,\kappa)$ of with $\dim(M) = \dim(R) = d$, then the following three conditions are equivalent for each $0 \leq k \leq d$: 
\begin{enumerate}
\item asydepth($M$) $\geq k$,
\item  $\sup\{\ell(H^i(x_1, \ldots, x_d;M))\mid \sqrt{f_1, \ldots, f_d} = m \mbox{, } i< k\}<\infty$,
\item $\ell(H^i_m(M)) < \infty$ for all $i<k$.
\end{enumerate} 
\end{theorem}

\noindent We will give two proofs of Theorem \ref{main}, one as Theorem \ref{ACMLCM} in Section 4 and one as Corollary \ref{forallk} in Section 5.  Recall that whenever $M$ is quasi-unmixed, i.e. if $\hat{M}$ is an equidimensional $\hat{R}$-module, condition (3) is equivalent to $\depth_{P}M_P \geq \mbox{height} (P)+k-d$ for all prime ideals $P \neq m$.  Note that when $k = d$, this is saying that $M$ is Cohen---Macaulay on the punctured spectrum.  It is surprising that requiring that certain Koszul cohomology modules grow slowly, as in the definition of aymptotic depth, actually forces the lengths of all Koszul cohomology modules on all systems of parameters to share a global bound.  That is to say that condition 2, which requires an absolute bound on the lengths of Koszul homology modules, is prima facie much stronger than condition 1, which merely requires that these lengths grow somewhat slowly.  

One key ingredient in the major proofs of this paper is the result that one can find prime ideals $P$ and sequences of parameter ideals $I_n$ such that $\ell(R/I_n)$ and $\ell(R/(I_n+P))$ have approximately the same length even when $\dim(R/P) = 1$.  The precise statements of these surprising results, together with their proofs, which are computational in nature, can be found in Section 6.  We give one example of such a sequence of parameter ideals here.

\begin{example}
Let $R = \kappa[[x,y,z]]$ where $k$ is any field, let $P = (x,y)$ and let \[
I_n = (z^{n^4}-z^nx^n,y^n-z^nx,x^{n+1}-xz^{n^4-n}+yz^n).
\]  Then $\ell(R/(I_n+(x,y))) = n^4$ while $n^4 \leq \ell(R/I_n) \leq (n^4+2n)+(2n+1)^2(3n)$, and so \[
\displaystyle \lim_{n \rightarrow \infty} \dfrac{\ell(R/I_n)}{\ell(R/(I_n+P))} = 1.
\]  
\end{example}  \noindent Details of the computation appear in Section 6 as part of Proposition \ref{basecase}.

We will use such parameter ideals together with one key spectral sequence argument to give a proof of the main result.  In Section 5, we will give an alternative proof that $M$ asymptotically Cohen---Macaulay implies $M$ Cohen---Macaulay on the punctured spectrum.  This alternative proof shows directly that $M$ asymptotically Cohen---Macaulay implies $M/H^0_m(M)$ torsion-free (Lemma \ref{torsionlemma}) and that the quotient of an asymptotically Cohen---Macaualy module by a non-zerodivisor is asymptotically Cohen--Macaulay (Lemma \ref{modx}), which gives intuition for asymptotic depth behaving as a depth theory.  This approach also allows us to give a rigidity result on asymptotic depth (Corollary \ref{rigidity}).  This argument gives more information than the Section 4 argument on where, meaning in which cohomology module, the obstruction to asymptotic Cohen--Macaulayness can be found.  In Corollary \ref{forallk}, we see how the equivalence of asymptotically Cohen--Macaulay with Cohen--Macaulay on the punctured spectrum and quasi-unmixed, together with the torsion-free result, implies directly the entirety of the main result.  

It is worth noting that the quasi-unmixed assumption is genuinely necessary for $M$ Cohen--Macaulay on the punctured spectrum to imply asymptotically Cohen--Macaulay (Theorem \ref{equidimnecessary}).  The example below gives intuition for that result.  Roughly speaking, what goes wrong in the example below is that the ratio of the length of the first Koszul cohomology module to the length of $R/I_n$ cannot be bounded away from $1$.  If it could be, given that $R$ is not quasi-unmixed, we would have a counterexample to \cite[Theorem 1]{Vogel}. 

\begin{example}
$R = \dfrac{\kappa[[x,y,z]]}{(xz, yz)}$.  Let $I_n = (x^n-z^{n^3}, y^n)$.  Then $e_{I_n}(R) = e_{(x^n, y^n)}(R/(z)) = n^2$ while $\ell(R/I_n) = \ell\left(\dfrac{\kappa[[x,y,z]]}{(y^n, x^n-z^{n^3}, xz, yz)}\right) = n^3+n(n+1)$, where the $n^3$ counts monomials with a positive power of $z$ whose image is nonzero in $R/I_n$, noting that $z^{n^3+1} = z(x^n) = 0 \in R/I_n$, and $n(n+1)$ counts monomials $y^ix^j$ with $0 \leq i<n$ and $0 \leq j<n+1$ whose images are nonzero in $R/I_n$, noting that $x^{n+1} = xz^{n^3} = 0 \in R/I_n$.  Using Serre's formula, $n^2 = e_{I_n}(R) = \ell(R/I_n) -\ell(H^1(x^n-z^{n^3}, y^n; R)) = n^3+n(n+1) -\ell(H^1(x^n-z^{n^3}, y^n; R)) $, and so $\ell(H^1(x^n-z^{n^3}, y^n; R)) = n^3+n$.  In particular, $\displaystyle \lim_{n \rightarrow \infty} \dfrac{\ell(H^1(x^n-z^{n^3}, y^n; R))}{\ell(R/I_n)} =  \lim_{n \rightarrow \infty} \dfrac{n^3+n}{n^3+n(n+1)} = 1 \neq 0$.  
\end{example}

\noindent\textbf{Acknowledgement}: The author would like to thank her advisor, Mel Hochster, for his tremendous support during her doctoral work, of which this paper is part.  The author has been partially supported by NSF Grants DMS-1401384 and 0943832.  

\section{Statements of results from Section 6 and some simplifications}

We summarize below in one statement the results whose proofs appear in Section 6.  We will need only Theorem \ref{basicsummarizedbighomology} in Section 4 to show the main result.  

\begin{theorem} \label{basicsummarizedbighomology}
If $R = \kappa[[x_1, \ldots, x_d]]$ or $R = V[[x_1,x_2, \ldots, x_{d-1}]]$ or $R = V[[x_1,x_2, \ldots, x_{d}]]/(p^sx_1)$ with $s \geq 1$, $d \geq 2$, and $V = (V,pV,\kappa)$ a  discrete valuation ring, then neither $R/(x_1, \ldots, x_k)$ nor $R/(p, x_1, \ldots, x_{k})$ nor $R/(p, x_2, \ldots, x_{k})$ with $1 \leq k \leq d-1$ are $d$-effaceable.  
\end{theorem}

This result is proved as Propositions \ref{d=2} and \ref{d=2mixed} in Section 6.  We will need the full strength of Theorem \ref{summarizedbighomology} in addition to Theorem \ref{basicsummarizedbighomology} in Section 5 to show that $M$ asymptotically Cohen--Macaulay implies $M/H^0_m(M)$ torsion-free.  

\begin{theorem} \label{summarizedbighomology}
Let $d \geq 3$, $k \geq 1$, $s \geq 1$,  $d-2 \geq h \geq 1$, $\kappa$ a field, and $V = (V, pV, \kappa)$ a discrete valuation ring.  If $R = \kappa[[x_1, \ldots, x_d]]$, then $M = (x_1, \ldots, x_{d-h})R$ is not $h+1$-effaceable.  If $R = V[[x_1,x_2, \ldots, x_{d-1}]]$, then neither $M = (p^k, x_1, \ldots, x_{d-(h+1)})R$ nor $M = (x_1, \ldots, x_{d-h})R$ is $h+1$-effaceable.  If $R = V[[x_1,x_2, \ldots, x_{d}]]/(p^sx_1)$, then none of $M = (p^k, x_1, \ldots, x_{d-h})R$ nor $M = (p^k, x_2, \ldots, x_{d-{(h-1)}})R$ nor $M = (x_1, \ldots, x_{d-(h-1)})R$ is $h+1$-effaceable.  In all cases above and also when $h = d-1$, $N = R/M$ is not $h$-effaceable.  
\end{theorem}

Theorem \ref{summarizedbighomology} combines the results of Theorem \ref{d>2}, Corollary \ref{d>2notprime}, Corollary \ref{d>2mixednotprime}, Proposition, \ref{d=2} and Proposition \ref{d=2mixed}.  

\begin{definition}
If $M$ is a finitely generated module over the local ring $(R, m, \kappa)$, let $H^i(I;M)$ denote any one of the $i^{th}$ Koszul cohomology modules of $(f_1, \ldots, f_r)$ on the module $M$ when $(f_1, \ldots, f_r)$ minimally generate $I$.
\end{definition}

Recall that if $f_1, \ldots, f_r$ and $g_1, \ldots, g_r$ both minimally generate $I$, then $H^i(f_1, \ldots, f_r;M) \cong H^i(g_1, \ldots, g_r;M)$ for all $i \in \Z$.  We will only use this notation when we are interested in the module $H^i(f_1, \ldots, f_r;M)$ up to isomorphism, for example when we are interested in its length, as we will frequently be.

We now make a reduction relevant to both Section 4 and Section 5.  We will reduce to the complete case.  Denote by $\hat{R}$ the $m$-adic completion of the local ring $(R, m, \kappa)$ and by $\hat{M}$ the $m$-adic completion of the $R$-module $M$.  Using the fact we recalled just above, it is sufficient to show that expansion and contraction give a bijection between ideals of $R$ generated by parameters on $M$ and those of  $\hat{R}$ generated by parameters on $\hat{M}$.  

\begin{lemma} \label{completereduction}
Let $M$ be a $d$-dimensional module over the local ring $(R, m, \kappa)$.  Then expansion and contraction give a bijection between ideals of $R$ generated by parameters on $M$ and ideals of $\hat{R}$ generated by parameters on $\hat{M}$.
\end{lemma}

\begin{proof}
By replacing $R$ by $R/\Ann(M)$, we may assume that $M$ is faithful and that $\dim(R) = d$.  Let $x_1, \ldots, x_d$ be parameters on $M$, and denote by $\hat{x_1}, \ldots, \hat{x_d}$ their images in $\hat{R} = (\hat{R}, \hat{m}, \kappa)$.  Because $M/(x_1, \ldots, x_d)M \cong \hat{M}/(\hat{x_1}, \ldots, \hat{x_d})\hat{M}$, it is immediate that $\hat{x_1}, \ldots, \hat{x_d}$ are parameters on $\hat{M}$.  Suppose that $\hat{y}, \ldots, \hat{y}_d \in \hat{R}$ are parameters on $\hat{M}$, and let $\hat{I} = (\hat{y_1}, \ldots, \hat{y_d})$.  Choose $N>0$ so that $\hat{m}^N \subseteq \hat{m}\hat{I}$.  We claim that for any $\hat{\varepsilon_i} \in \hat{m}^N$ for $1 \leq i \leq d$, $\hat{I} = (\hat{y_1}+\hat{\varepsilon_1}, \ldots, \hat{y_d}+\hat{\varepsilon_d})$.  The containment $ (\hat{y_1}+\hat{\varepsilon_1}, \ldots, \hat{y_d}+\hat{\varepsilon_d}) \subseteq \hat{I}$ is clear because each $\hat{y_i} \in \hat{I}$ and each $\hat{\varepsilon_i} \in \hat{m}^N \subseteq \hat{m}\hat{I} \subseteq \hat{I}$.  The other containment follows from Nakayama's Lemma together with the fact that images of the $\hat{y_i}+\hat{\varepsilon_i}$ generate $\hat{I}/\hat{m}\hat{I}$ because the $\hat{\varepsilon_i} \in \hat{m}\hat{I}$ and the $\hat{y_i}$ generate $\hat{I}$.  Now choose the $\hat{\varepsilon_i}$ so that each $\hat{y_i}+\hat{\varepsilon_i}$ is the image of some $y_i \in R$, and set $I = (y_1, \ldots, y_d)$.  Because $I/mI \cong \hat{I}/\hat{m}\hat{I}$, we have, as the notation suggests, $\hat{I} = I\hat{R}$, and, using the well-known bijection between $m$-primary ideals of $R$ and of $\hat{m}$-primary ideals of $\hat{R}$, $I = \hat{I} \cap R$.  
\end{proof} 

\section{The spectral sequence argument in all characteristics}

The purpose of this section is to prove Theorem \ref{main}.  Let $(\underline{\hspace{0.5cm}})^\vee$ denote the Matlis dual.
 
 \begin{lemma} \label{SpecSeqLemma} Let $(R,m,\kappa)$ be a complete local ring, and let $M$ be a finitely generated $R$-module of dimension $d$. Then for each $0 \leq i <d$ and every ideal $I = (f_1, \ldots, f_d)$ of $R$ generated by parameters on $M$, \[
 \ell(H^i(f_1, \ldots, f_d;M)) \leq \sum_{j = 0}^{i} \ell(H^{d-i+j}(f_1, \ldots, f_d;H^j_m(M)^\vee)). \label{goodhomologybound}
 \]
\end{lemma}

\begin{proof}
The argument is similar to Theorem 3.16 in \cite{sixLectures}, and an explanation of how this result follows directly from the that result can be found in \cite{joint}.  We consider the spectral sequence of the double complex coming from applying $\Hom( \underline{\hspace{5mm}}, \mathcal{A}^\bullet)$, where $\mathcal{A}^\bullet$ is the dualizing complex of $R$, with the homological Koszul complex of $(f_1, \ldots, f_d)$ on $M$.  The inequality in the statement of the theorem can be read off from the $E_2$ page.
\end{proof}
      
 \begin{theorem}\label{backward} Let $(R,m,\kappa)$ be a local ring, let $M$ be a finitely generated $R$-module of dimension $d$, and fix $0 < k \leq d$. If $\ell(H_m^j(M))<\infty$ for all $j<k$, then \[
 \sup\{\ell(H^i((f_1, \ldots, f_d);M)) \mid f_1, \ldots, f_d \mbox{ parameters on $M$}, 0 \leq i <k\} <\infty.
 \]  Moreover, if $R$ is complete, then for each $0 \leq i<k$, \[
 \sup\{\ell(H^i((f_1, \ldots, f_d);M)) \mid f_1, \ldots, f_d \mbox{ parameters on $M$}\} \leq \sum_{j = 0}^{i} {d \choose i-j} \ell(H^j_m(M)^\vee).
 \]
 \end{theorem}
 
 \begin{proof}
Using Lemma \ref{completereduction}, we may assume that $R$ and $M$ are complete.  Therefore, it suffices to show only the last statement of the theorem, which follows from Lemma \ref{SpecSeqLemma}.  More precisely, by examining the Koszul cohomology modules $H^{d-i+j}(f_1, \ldots, f_d;H^j_m(M)^\vee)$ on the right-hand side of the inequality in Lemma \ref{SpecSeqLemma}, using the facts that each  $\ell(H_m^j(M))<\infty$ for $0 \leq j <d-1$ and that $(f_1, \ldots, f_d) \subseteq m$, we see that for each parameter ideal $I = (f_1, \ldots, f_d)$ and each $0 \leq i <d$, \[
  \ell(H^i(f_1, \ldots, f_d;M)) \leq \sum_{j = 0}^{i} {d \choose i-j} \ell(H^j_m(M)^\vee).
  \]
  \end{proof}
  
 We now aim to prove that asymptotically small Koszul cohomology implies finite length local cohomology.  We begin by reducing to the case of a ring that is either regular or Gorenstein of a certain form.  
  
 \begin{lemma} \label{RegGorReduction} Suppose that for all finitely generated modules $M$ over a complete local ring $(S,n)$ with $\dim(M) = \dim(S)$ where $S$ is either regular or Gorenstein of the form $\dfrac{V[[x_1, \ldots, x_d]]}{(p^sx_1)}$, where $p$ generates the maximal ideal of the complete discrete valuation ring $V$ and $s \geq 1$, if asydepth $(M) \geq k$, then $H^i_m(M)<\infty$ for all $i<k$.  Then the same holds over any local ring $(R, m)$.
\end{lemma} 

\begin{proof} Using Lemma \ref{completereduction}, we may assume that $R$ and $M$ are complete.  Suppose asydepth $(M) \geq k$ over the complete local ring $R$.  By Cohen's structure theorem, $R$ is a module-finite extension of a ring $(S,n)$ that is either regular or of the form $S = \dfrac{V[[x_1, \ldots, x_d]]}{(p^sx_1)}$, where $p$ generates the maximal ideal of the complete discrete valuation ring $V$ and $s \geq 1$.  Every parameter ideal $I_n$ of $S$ is a parameter ideal of $R$, every finitely generated $R$-module is also a finitely generated $S$-module, and $H^i_R(I_n;M) = H^i_S(I_n;M)$ for all $i \in \Z$.  Because $\ell(R/I_nR) \leq \nu_S(R) \cdot \ell(S/I_n)$ for each $n \geq 1$, where $\nu_S(R)$ denotes the minimal number of generators of $R$ as an $S$-module, we have $\dfrac{\ell(H^i(I_n;M))}{\ell(S/I_n)} \leq \nu_S(R) \cdot \dfrac{\ell(H^i(I_n;M))}{\ell(R/I_nR)} \xrightarrow{n \rightarrow \infty} 0$ for all $i <k$, which is to say that asydepth $(M) \geq k$ over $S$.  By assumption, then, $\ell(H^i_n(M))<\infty$ for all $i<k$.  But $H^i_n(M) = H^i_m(M)$ for all $0 \leq i \leq d$, and so $\ell(H^i_m(M))<\infty$ for all $i<k$.
\end{proof}

\begin{theorem} \label{k=d-1}
Let $(R, m, \kappa)$ be a complete $d$-dimensional ring with that is either regular or Gorenstein of the form $\dfrac{V[[x_1, \ldots, x_d]]}{(p^sx_1)}$ where $V$  is a complete discrete valuation ring with maximal ideal $(p)$ and $s \geq 1$.  Let $M$ be a finitely generated $R$-module.  Then for all $0 \leq k \leq d$, asydepth($M$) $\geq k$ implies $\ell(H^i_m(M)) < \infty$ for all $i <k$.
\end{theorem}

\begin{proof}
Suppose that asydepth($M$) $\geq k$ but that $\ell(H^i_m(M)) = \infty$ for some $0 \leq k \leq d$, and assume that $k$ is minimal with respect to this property.   Notice that for each parameter ideal $I$, $H^{i+1}(I;\syz^1(M)) \cong H^i(I;M)$ and $H^{i+1}(\syz^1(M)) \cong H^i_m(M)$ for each $0 \leq i < d-1$, and so by replacing $M$ by $\syz^{d-k-1}(M)$, we have a counterexample when $k = d-1$ and $\ell(H^i_m(M))<\infty$ for all $0 \leq i <d-1$.  We now return to the spectral sequence from Lemma \ref{SpecSeqLemma} in order to improve the inequality.  Below is the possibly nonzero component of the $E_2$ page of the spectral sequence run by taking homology of columns first.  We omit maps, which will not be of interest to us. $$\CD 
@.               @.                                              @.                      @.              @.\\
@. H^0(I;\mbox{H}^d_n(M)^\vee) \hspace{0.7cm}@. H^0(I;\mbox{H}^{d-1}_n(M)^\vee) \hspace{0.7cm}@.\cdots @.\cdots @.\cdots @.\hspace{0.7cm} H^0(I;\mbox{H}^0_n(M)^\vee) @. \\ 
@. H^1(I;\mbox{H}^d_n(M)^\vee) \hspace{0.7cm}@. H^1(I;\mbox{H}^{d-1}_n(M)^\vee) \hspace{0.7cm}@.\cdots @. \cdots @.\cdots @.\hspace{0.7cm} \mathbf{H^1(I;\mbox{H}^0_n(M)^\vee)} @. \\ 
@.               @.                                               @.                      @.             @.\\
{}     @.        \vdots                                      @.        \vdots   @.     \vdots   @.\vdots   @.\vdots   @.    \vdots                      @. {} \\
@.               @.                                               @.                      @.             @.\\                
@. H^i(I;\mbox{H}^d_n(M)^\vee) \hspace{0.7cm}@. H^i(I;\mbox{H}^{d-1}_n(M)^\vee) \hspace{0.7cm}@.\cdots @.  \hspace{0.7cm} \mathbf{H^i(I;\mbox{H}^{i-1}_n(M)^\vee)}  \hspace{1cm} @.\cdots @.\hspace{1cm} H^i(I;\mbox{H}^0_n(M)^\vee) @. \\
@.               @.                                               @.                      @.              @.\\
{}     @.        \vdots                                      @.        \vdots   @.       \vdots   @.\vdots   @.\vdots   @.  \vdots                      @. {} \\
@.               @.                                               @.                      @.              @.\\                
@. H^d(I;\mbox{H}^d_n(M)^\vee) \hspace{0.7cm}@. \mathbf{H^d(I;\mbox{H}^{d-1}_n(M)^\vee)} \hspace{0.7cm}@.\cdots @.\cdots @.\cdots @. \hspace{0.7cm} H^d(I;\mbox{H}^0_n(M)^\vee) @. \\
@.               @.                                              @.                      @.              @.\\
      \endCD$$  
      
We may from here improve the result of Lemma \ref{goodhomologybound} in the case of $i = d-1$.  Note that $\ell(H^{d-1}(I;M))$ is computed as the sum of the lengths of modules on the $E_{\infty}$ page currently occupied in the grid above by the modules $H^{j}(I; H^{j-1}_m(M))$ for $0 \leq j \leq d$.  The appropriate diagonal is shown in bold above.  For $j<d$, we already have $\ell(H^{j-1}_m(M))<\infty$.  It follows that $H^d(I;H^{d-1}_m(M))$ only maps to and from modules of finite length on pages $E_i$ for $i>2$.  Therefore, controlling the lengths of the $H^{d-1}(I;M)$ is the same task as controlling the lengths of the $H^d(I;H^{d-1}_m(M))$.  More precisely, the task remaining to us is to show that if $\dim(H^{d-1}_m(M)) >0$, then there exists a sequence of parameter ideals $I_n$ such that $\dfrac{\ell(H^{d-1}_m(M)/I_nH^{d-1}_m(M))}{\ell(R/I_n)} \not\rightarrow 0$ as $I_n \rightarrow \infty$.  If $\dim(H^{d-1}_m(M)) >0$, then there is a map $H^{d-1}_m(M) \twoheadrightarrow R/P$ for some prime ideal $P$ with $\dim(R/P)>0$.  It is, therefore, sufficient to find a sequence of parameter ideals $I_n$ such that $\dfrac{\ell(R/(P+I_n))}{\ell(R/I_n)} \not\rightarrow 0$ as $I_n \rightarrow \infty$ for every prime ideal $P$ with $\dim(R/P)>0$.  

We now seek to reduce to the special cases of the above limits studied in Section 6.  Suppose that $P$ is height $h$.  Choose $y_1, \ldots, y_h \in R$ such that $\frac{y_1}{1}, \ldots, \frac{y_h}{1}$ form a system of parameters in $R_P$, and, using prime avoidance, extend to a system of parameters $y_1, \ldots, y_d$ in $R$.  In the mixed characteristic case in which $p$ is a parameter in $R$, choose $y_1 = p$ if $p \in P$ or $y_d = p$ if $p \notin P$.  (When extending to a full system of parameters, choose $y_d$ before choosing $y_j$ for $h <j<d$.)  Form the ring $S = \kappa[[y_1, \ldots, y_d]]$ in the equal characteristic case, $S = V[[y_2, \ldots, y_d]]$ in the mixed characteristic case with $p \in P$ and $p$ a parameter, and $S = V[[y_1, \ldots, y_{d-1}]]$ if $p \notin P$ and $p$ a parameter.  If $p$ is not a parameter in $R$ and $(x_1, p) \subseteq P$, take $y_1 = x_1-p$, and form $S = \dfrac{V[[x_1, y_2, \ldots, y_d]]}{(p^sx_1)}$.  If $(x_1, p) \not\subseteq P$, take $y_{d} = x_1-p$, and form $S = \dfrac{V[[x_1, y_1, y_2, \ldots, y_{d-1}]]}{(p^sx_1)}$.  (Again, choose $y_d$ before the $y_j$ for $h <j<d$ when necessary.)

In all cases we have the short exact sequence $0 \rightarrow S \rightarrow R \rightarrow C \rightarrow 0$ with $\ell(C)<\infty$.  Now for any parameter ideal $I$ of $S$, we have from the long exact sequence of Koszul homology, \[
0 \rightarrow H^{d-1}(I;C) \rightarrow S/I \rightarrow R/IR \rightarrow C/IC \rightarrow 0.
\]  Note that $\ell(H^{d-1}(I;C))< {d \choose 2} \ell(C)$ and $\ell(C/IC)<\ell(C)$ independent of $I$.  It follows that $\displaystyle \lim_{I_n \rightarrow \infty}\dfrac{\ell(R/I_nR)}{\ell(S/I_n)} \rightarrow 1$ for every sequence of parameter ideals $I_n \subseteq m^n$.

Now if $S$ is regular, set $Q = (y_1, \ldots, y_h) \subseteq S$.  In the Gorenstein case, set $Q = (p, y_1, \ldots, y_h)$ or $Q = (p, x_1, y_1, \ldots, y_h)$ or $Q = (x_1, y_1, \ldots, y_h)$ depending on whether $p \in P$ but $x_1 \notin P$, $p \in P$ and $x_1 \in P$, or $p \notin P$ but $x_1 \in P$, respectively.   Using the short exact sequence $0 \rightarrow S/Q \rightarrow R/P \rightarrow \bar{C} \rightarrow 0$, where $\bar{C}$ is a quotient of $C$ completing the short exact sequence, and the argument from the proceeding paragraph, we have $\displaystyle \lim_{I_n \rightarrow \infty}\dfrac{\ell(R/(I_n+P)R)}{\ell(S/(I_n+Q))} \rightarrow 1$ for every sequence of parameter ideals $I_n \subseteq m^n$.  We may, therefore, assume that $R$ is either regular or Gorenstein of the form described.  The result now follows from Propositions \ref{d=2} and \ref{d=2mixed}.  
\end{proof}

We are now prepared to prove the desired theorem.

\begin{theorem} \label{ACMLCM} If $M$ is a finitely generated module over the local ring $(R,m,\kappa)$ of with $\dim(M) = \dim(R) = d$, then the following three conditions are equivalent for each $0 \leq k \leq d$: 
\begin{enumerate}
\item asydepth($M$) $\geq k$,
\item  $\sup\{\ell(H^i(f_1, \ldots, f_d;M))\mid \sqrt{f_1, \ldots, f_d} = m \mbox{, } i< k\}<\infty$,
\item $\ell(H^i_m(M)) < \infty$ for all $i<k$.
\end{enumerate} 
\end{theorem}

\begin{proof}
It is obvious that (2) implies (1).  Using Lemma \ref{completereduction}, we may assume that $R$ and $M$ are complete.  Using Lemma \ref{RegGorReduction}, we may assume that $R$ is a complete ring with that is either regular or Gorenstein of the form $\dfrac{V[[x_1, \ldots, x_d]]}{(p^sx_1)}$ where $V$  is a complete discrete valuation ring with maximal ideal $(p)$ and $s \geq 1$.  The result now follows from Theorem \ref{backward} and Theorem \ref{k=d-1}.
\end{proof}

\begin{example}\label{badmultiplicity} We now give an application of Theorem \ref{ACMLCM}.  In particular, we give an example of a local ring $(R, m)$ showing that $\dfrac{\ell(R/I_n)}{e_{I_n}(R)}$ need not approach 1 as $n \rightarrow \infty$ for every sequence of parameter ideals $I_n \subseteq m^n$.  This example shows that Lech's limit formula does not generalize to the case of all parameter ideals in increasingly high powers of the maximal ideal.  Let $(R,m)$ be the localization of $(\frac{\kappa[x,y,z]}{x^3+y^3+z^3} \circledS \kappa[u,v])[w]$ at the homogeneous maximal ideal, where $\kappa$ is a field, and $\circledS$ denotes Segre product.  Because $R$ is normal, it is in particular $S_2$, and so its only nonvanishing Koszul cohomology modules are $i = 3,4$.  It is not Cohen--Macaulay on the punctured spectrum, and so $\ell(H^i_m(R)) = \infty$ for some $i<d$, and so $R$ is not asymptotically Cohen--Macaulay by Theorem \ref{ACMLCM}.  Therefore, we may pick a sequence of parameter ideals $I_n$ such that $\dfrac{H^3(I_n;R)}{\ell(R/I_n)} \not\rightarrow 0$.  It follows that for that sequence $\dfrac{e_{I_n}(R)}{\ell(R/I_n)} = \dfrac{H^4(I_n;R)-H^3(I_n;R)}{{H^4(I_n;R)}} \not\rightarrow 1$.
\end{example}

Because we do an explicit computation of the only non-zero homology modules below, we do not exactly think of this example as an application of the theorem.  But the computation is certainly part of our understanding of the theorem, and it is a particularly interesting example of the computation of the ratio $\dfrac{e_{I_n}(M)}{\ell(M/I_nM)}$.  
\begin{example}
Let $R = \kappa[[x,y,z]]$ and $M = (x, y)R$.  Note that depth$_m(M) = 2$, and so $H^i(I;M) = 0$ for $i = 0, 1$ for every parameter ideal $I$.  Because $M$ is generated by 2 elements, $\ell(M/IM) \leq 2 \cdot \ell(R/I)$ for every $m$-primary ideal $I$.  Let $I_n = (z^{n^4}-z^nx^n,y^n-z^nx,x^{n+1}-xz^{n^4-n}+yz^n)$, as in Proposition \ref{basecase}.  Using the computation in Proposition \ref{basecase}, we compute \begin{align*}
0 &\leq \lim_{n \rightarrow \infty} \dfrac{e_{I_n}(M)}{\ell(M/I_nM)} =  \lim_{n \rightarrow \infty}\dfrac{\ell(M/I_nM)-\ell(H^2(I_n;M))}{\ell(M/I_nM)} = 1-\lim_{n \rightarrow \infty}\dfrac{\ell(H^2(I_n;M))}{\ell(M/I_nM)} \\&
\leq 1-\lim_{n \rightarrow \infty}\dfrac{\ell(H^2(I_n;M))}{2 \cdot \ell(R/I_n)} = 1- \lim_{n \to\infty}  \dfrac{n^4}{2(n^4+2n+(2n+1)^2(3n))}  = \dfrac{1}{2}.
\end{align*}
\end{example}

We pose two questions for further study about possible generalizations of Lech's limit formula and related asymptotic vanishing of Koszul cohomology.

\begin{question}
For a fixed local ring $(R, m, \kappa)$ and finitely generated $R$-module $M$, which sequences of parameter ideals $I_n \subseteq m^n$ satisfy $\displaystyle \lim_{n \rightarrow \infty}\dfrac{\ell(e_{I_n}(M))}{\ell(M/I_nM)} = 1$?
\end{question}

\begin{question}
For a fixed local ring $(R, m, \kappa)$ and finitely generated $R$-module $M$ with $\dim(R) = \dim(M)$, which sequences of parameter ideals $I_n \subseteq m^n$ satisfy $\displaystyle \lim_{n \rightarrow \infty}\dfrac{\ell(H^i(I_n;M))}{\ell(R/I_n)} = 0$ for all $0 \leq i <\dim(M)$?
\end{question}
      
\section{An alternative approach} 

We will begin by showing that if a module $M$ is asymptotically Cohen--Macaulay, then $M$ is Cohen--Macaulay on the punctured spectrum.  Our first step will be to show that if $M$ is asymptotically Cohen--Macaulay, then $M/H^0_m(M)$ must be torsion-free.  Once we have that $M/H^0_m(M)$ is torsion-free, we show that we may quotient by a non-unit non-zerodivisor and preserve the asymptotically Cohen--Macaulay property, at which point we are prepared to show that $M$ is Cohen--Macaulay on the punctured spectrum by induction.  Lastly, we will show by taking syzygies over an appropriately chosen Cohen--Macaulay ring that this equivalence gives all of Theorem \ref{main}.  

\begin{definition}
Let $R$ be a ring and $M$ an $R$-module.  We call $M$ \emph{torsion-free} if, for each $m \in M$ and $r \in R$, whenever $rm = 0$, $m = 0$ or $r \in P$ for some $P \in \Ass(R)$.  
\end{definition}

\begin{lemma} \label{torsionlemma} Suppose $(R,m,\kappa)$ is a complete regular local ring of dimension  $d \geq 2$ or a Gorenstein ring of the form $V[[x_1, \ldots, x_{d}]]/(p^sx_1)$ with $d \geq 2$ and $s \geq 1$ where $V = (V, pV, \kappa)$ is a complete discrete valuation ring. If $M$ is dimension $d$ and asymptotically Cohen--Macaulay, then $M/H^0_m(M)$ is torsion-free. 
\end{lemma}

\begin{proof} Because $H^0_m(M)$ is finite length, either $M$ and $M/H^0_m(M)$ are both asymptotically Cohen--Macaulay or neither is because the difference between any $\ell(H^i(I_n;M))$ and $\ell(H^i(I_n;M/H^0_m(M)))$ is bounded by $({d \choose i}+{d \choose i-1}) \cdot \ell(H^0_m(M))$ from the long exact sequence of Koszul cohomology while $\ell(R/I_n) \rightarrow \infty$.  We may, therefore, replace $M$ by $M/H^0_m(M)$.  Suppose $T \neq 0$ is the torsion submodule of $M$.  $T$ cannot be supported only at the maximal ideal because then $T \subseteq H^0_m(M) = 0$.  Let $P \subseteq R$, a prime ideal of height $h<d$, be minimal in Supp($T$).  We will show that $M$ is not $(d-h)$-effaceable.  

First suppose that $R$ is regular, in which case $R = \kappa[[x_1, \ldots, x_d]]$ or $R = V[[x_2, \ldots, x_{d}]]$ where $V = (V,p,k)$ is a complete discrete valuation ring.    Let $(y_1, \ldots, y_h)$ or $(p, y_1, \ldots, y_{h-1})$ be local generators of $P$ that extend to a system of parameters $y_1, \ldots, y_d$ or $p, y_2, \ldots, y_{d}$ of $R$, and form the regular ring $S = \kappa[[y_1, \ldots, y_d]]$ if $R$ is equal characteristic or $S = V[[y_2, \ldots, y_{d}]]$ if $R$ is mixed characteristic. (By clearing denominators, we may without loss of generality assume that the local generators of $P$ are elements of $R$.)  This can be done in equal characteristic by prime avoidance by choosing $y_1$ a minimal generator of $P$ and each $y_i$ for $2 \leq i \leq h$ a minimal generator of $P$ not in the minimal primes of $(y_1, \ldots, y_{i-1})$.  In the mixed characteristic regular ring case, we choose each $y_i$ to avoid $(p, y_2, \ldots, y_{i-1})$ and the minimal primes of $((p, y_2, \ldots, y_{i-1})+P^2)R_P \cap R$ for $2 \leq i \leq h$ so that $p$ will also be a parameter in $S$.  In the Gorenstein case, as in Theorem \ref{k=d-1}, we will find appropriate $y_1, \ldots, y_h$ to be a system of parameters of $R_P$.  If $(x_1, p) \subseteq P$, take $y_1 = x_1-p$, choose the remaining $y_i$ by prime avoidance as above, and form $S = \dfrac{V[[x_1, y_2, \ldots, y_d]]}{(p^sx_1)}$.  If $(x_1, p) \not\subseteq P$, take $y_{d} = x_1-p$, choose the remaining $y_i$ by prime avoidance, and form $S = \dfrac{V[[x_1, y_1, y_2, \ldots, y_{d-1}]]}{(p^sx_1)}$. Because $P$ was a minimal prime of $T$, we have $S/(P\cap S) \hookrightarrow R/P \hookrightarrow T \hookrightarrow M$.  More concretely, we have $\bar{S} = S/(P \cap S) = \kappa[[y_{h+1}, \ldots, y_d]] \hookrightarrow M$ or $\bar{S} = S/(P \cap S) = V[[y_{h+1}, \ldots, y_{d-1}]]$.  Note that in the Gorensein case, we must have $x_1 \in P$ or $p \in P$ so that, by our construction, $S/(P \cap S)$ will always be regular.  

We aim to use these injections to split off, as a direct summand of $M$ over a smaller regular or Gorenstein ring, a torsion module of the form studied in Section 6.  Let $M'$ be a maximal $S$-submodule of $M$ disjoint from $\bar{S}$ and $N=M/M'$.  Then $\bar{S} \hookrightarrow N$ is an essential extension, and a retraction of the inclusion of $\bar{S}$ into $N$ lifts to a retraction of the map to $M$.  

We begin with the equal characteristic case.  For each $i \leq h$, there exists $k \geq 1$ such that $(y_i)^kN \cap \kappa[[y_{h+1}, \ldots, y_d]] = (y_i)^{n-k}(y_i^kN \cap \kappa[[y_{h+1}, \ldots, y_d]]) = 0$ by the Artin-Rees Lemma.  But $(y_i)^kN \cap \kappa[[y_{h+1}, \ldots, y_d]] = 0$ implies that $(y_i)^kN = 0$ because the extension is essential.  Therefore, after replacing each $y_i$ with $y_i^{k_i}$, we may assume without loss of generality that $y_iN = 0$ for all $i \leq h$ and view $N$ as a module over $\kappa[[y_{h+1}, \ldots, y_d]]$.  Now because Frac($\kappa[[y_{h+1}, \ldots, y_d]]$) is a maximal essential extension of $\kappa[[y_{h+1}, \ldots, y_d]]$, we may view $N$ as a finitely generated submodule of Frac($\kappa[[y_{h+1}, \ldots, y_d]]$), i.e. $N \cong \kappa[[y_{h+1}, \ldots, y_d]][\frac{1}{f}]$ for some $f \in \kappa[[y_{h+1}, \ldots, y_d]]$.  Equivalently, the essential extension we have been studying may be described as $\kappa[[y_{h+1}, \ldots, y_d]] \xrightarrow{f} \kappa[[y_{h+1}, \ldots, y_d]]$.  Choosing $t\gg0$, we may take $f$ to be part of a basis of $\kappa[[y_{h+1}, \ldots, y_d]]$ over $\kappa[[y_{h+1}^t, \ldots, y_d^t]]$, which means that our map splits as a map of $\kappa[[y_{h+1}^t, \ldots, y_d^t]]$-modules and so as a map of $A = \kappa[[y_1, \ldots, y_h, y_{h+1}^t, \ldots, y_d^t]]$-modules.  

In mixed characteristic, we consider the case of $p \in P$ and $p \notin P$ separately.  We first assume $p \notin P$. Call $\bar{A} = V[[y_{h+1}, \ldots, y_{d-1}]]$.  Because $\bar{A}$ injects into $M$, $\bar{A}_{(p)}$ injects into $M_{(p)}$.  Because $\bar{A}_{(p)}$ is a discrete valuation ring, $M_P$ must be free over $\bar{A}_{(p)}$, and so we may choose an element $u$ of a free basis of $M_{(p)}$ over $\bar{A}_{(p)}$ and note that $u \notin pM_{(p)}$.  Then $\bar{A}_{(p)}\hookrightarrow M_{(p)}$ given by $1 \mapsto u$ is a splitting.  Now because $\Hom(M_{(p)}, \bar{A}_{(p)}) \cong \Hom(M, \bar{A})_{(p)}$, the retraction $M_{(p)} \twoheadrightarrow \bar{A}_{(p)}$ with $u \rightarrow 1$ gives a map $\alpha:M \rightarrow \bar{A}$ with $\alpha(u) \notin (p)\bar{A}$.  Therefore, there exists a map $\theta: M \hookrightarrow F$ where $F$ is a free $\bar{A}$ module and $\theta(u) \notin (p)F$.  Now because $\bigcap_t (p, y_{h+1}^t, \ldots, y_{d-1}^t)F = (p)F$ and $\theta(u) \notin (p)F$, we may choose $t$ sufficiently large that $\theta(u) \notin (p, y_{h+1}^t, \ldots, y_{d-1}^t)F$, which is to say that $\theta(u)$ is not in the maximal ideal of $B = V[[y_{h+1}^t, \ldots, y_{d-1}^t]]$ expanded to $F$.  Because $F$ is free over $\bar{A}$ and $\bar{A}$ is free over $B$, $F$ is free over $B$.  It follows that there is a retraction $F \twoheadrightarrow \bar{A}$ as $B$ modules.  Composing with a retraction $\bar{A} \twoheadrightarrow B$ and restriction to $M$, we obtain a splitting of $B \hookrightarrow M$ as $B$-modules, which gives a splitting as $A = \dfrac{V[[x_1, y_1, \ldots, y_h, y_{h+1}^t, \ldots, y_{d-1}^t]]}{(p^sx_1)}$-modules.  

Lastly, we suppose $p \in P$, in which case $\bar{A} = \kappa[[y_{h+1}, \ldots, y_{d}]]$, thinking of $y_d = x_1-p$ when $x_1 \notin P$.  Fix $k$ so that $p^kM = 0$ but $p^{k-1}M \neq 0$.  As in the  previous cases, we replace each $y_i$ with some $y_i^{k_i}$ for $2 \leq i < h$ so that for each such $i$, $y_iM = 0$, and think of $M$ as a module over $B = \dfrac{V[[x_1, y_{h+1}, \ldots, y_{d-1}]]}{(p^k, p^sx_1)}$ in the Gorenstein case if $x_1 \notin P$ or over $B = \dfrac{V}{(p^k)}[[y_{h+1}, \ldots, y_{d}]]$ in the regular case or in the Gorenstein case if $x_1 \in P$.  For each $t \geq 0$, set $B_t = \dfrac{V}{(p^k)}[[y_{h+1}^t, \ldots, y_{d}^t]]$ or $B_t = \dfrac{V[[x_1^t, y_{h+1}^t, \ldots, y_{d-1}^t]]}{(p^k, p^sx_1^t)}$ as is appropriate to the case.  We aim to find a $t$ such that a copy of $B_t$  splits from $M$ as a $B_t$-module.   Because $B_{(p)}$ is a $0$-dimensional Gorenstein ring, it splits from $M_{(p)}$.  As in the previous case, this gives a map $\alpha:M \to B$ with an element $u \in M$ such that $\alpha(u) \notin (p)B$. Again, choose $t$ sufficiently  large that $\alpha(u) \notin (p, y_{h+1}^t, \ldots, y_{d}^t)B$ or $\alpha(u) \notin (p, x_1^t, y_{h+1}^t, \ldots, y_{d-1}^t)B$ as is appropriate to the case.   Now $B$ is free over $B_t$  and $\alpha(u)$ is not in the expansion of the maximal ideal of $B_t$ to $B$, so there is a $B_t$ module map $B \to B_t$ such that $\alpha(u) \mapsto 1$, and so the composite map  $M \to B \to B_t$  sends $u$ to $1$, which gives a splitting of $B_t$ from $M$ as a $B_t$-module.  This map is also a splitting over $A = V[[y_2, \ldots, y_h, y_{h+1}^t, \ldots, y_d^t]]$ in the regular case, over $A = \dfrac{V[[x_1, y_2, \ldots, y_h, y_{h+1}^t, \ldots, y_d^t]]}{(p^sx_1)}$ in the Gorenstein case when $x_1 \in P$, and over $A = \dfrac{V[[x_1^t, y_2, \ldots, y_h, y_{h+1}^t, \ldots, y_{d-1}^t]]}{(p^sx_1^t)}$ in the Gorenstein case when $x_1 \notin P$.

We now have a module of the form of of Theorem \ref{d>2} or Proposition \ref{d=2} as a direct summand of $M$ over a Gorenstein ring, which have named $A$ in each case.  Because $R$ is module finite over $A$, it is sufficient to find a sequence of parameter ideals $I_n$ in $A$ such that $\displaystyle \lim_{n \to\infty} \dfrac{\ell(H^{d-h}(I_n;M))}{\ell(R/I_nR)} \neq 0$.  Because $\ell(R/I_nR) \leq \nu_A(R) \cdot \ell(A/I_n)$, it is sufficient to show that $\displaystyle \lim_{n \to\infty} \dfrac{\ell(H^{d-h}(I_n;M))}{\ell(A/I_nA)} \neq 0$.  But $M \cong \bar{A} \oplus_A N$ or $M \cong B_t \oplus_A N$ for some $A$-module $N$, as is appropriate to the case, and Koszul homology splits over direct sums, so, for each $n \geq 1$, $\dfrac{\ell(H^{d-h}(I_n;M))}{\ell(A/I_n)}  \geq \dfrac{\ell(H^{d-h}(I_n;\bar{A}))}{\ell(A/I_n)}$ or $\dfrac{\ell(H^{d-h}(I_n;M))}{\ell(A/I_n)}  \geq \dfrac{\ell(H^{d-h}(I_n;B_t))}{\ell(A/I_n)}$, as is appropriate to the case.  But by Theorem \ref{d>2} and Proposition \ref{d=2},  $\displaystyle \lim_{n \to\infty} \dfrac{\ell(H^{d-h}(I_n;\bar{A}))}{\ell(A/I_n)} \neq 0$ and  $\displaystyle \lim_{n \to\infty} \dfrac{\ell(H^{d-h}(I_n;B_t))}{\ell(A/I_n)} \neq 0$, and so $\displaystyle \lim_{n \to\infty} \dfrac{\ell(H^{d-h}(I_n;M))}{\ell(A/I_n)} \neq 0.$  That is to say that $M$ is not $(d-h)$-effaceable and, in particular, is not asymptotically Cohen--Macaulay, a contradiction.

\end{proof} 

\begin{lemma}\label{modx} 
Let $(R,m,\kappa)$ be a Cohen--Macaulay local ring, $M$ a finitely generated $R$-module with $\dim(M) = \dim(R)$, and $x$ a non-unit of $R$ not a zerodivisor on $M$.  If $M$ is asymptotically Cohen--Macaulay over $R$, then $M/xM$ is asymptotically Cohen--Macaulay over $(R/(x),\mu, \kappa)$.  
\end{lemma}

\begin{proof}
Fix $R$, $M$ and $x$ as in the theorem statement.  By $\ref{completereduction}$, we may assume that $R$ and $M$ are complete.  Using the long exact sequence for Koszul cohomology coming from the short exact sequence $0 \rightarrow \syz^1(M) \rightarrow R^h \rightarrow M \rightarrow 0$, we have $0 \rightarrow H^i(I;M) \rightarrow H^{i+1}(I;\syz^1(M)) \rightarrow 0$ for every $i <d-1$ and every parameter idea l$I$ of $R$.  It follows that $\syz^1(M)$ is asymptotically Cohen--Macaulay whenever $M$ is.  Similarly, from $0 \rightarrow \syz^1(M/xM) \rightarrow (R/(x))^g \rightarrow M/xM \rightarrow 0$ with $g \leq h$, we have $0 \rightarrow H^i(J;M/xM) \rightarrow H^{i+1}(J;\syz^1(M/xM)) \rightarrow 0$ for all $i<d-2$ and every parameter ideal $J$ of $R/(x)$.  It follows that if $M/xM$ is not $\noi$ for some $i<d-2$, then $\syz^1(M/xM) \cong \syz^1(M)/x \cdot \syz^1(M)$ is not $i+1$-effaceable.  But $\syz^1(M)$ must be asymptotically Cohen--Macaulay because $M$ is.  Therefore, we may assume by induction that $i = d-2$. 

Let $\varepsilon>0$.  We aim to show that there exists $N' \in \N$ such that for all parameter ideals $I' \subseteq \mu^{N'}$, $\dfrac{\ell(H^{d-2}(I';M/xM))}{\ell((R/(x))/I'(R/(x)))} < \varepsilon$.  We claim that $\syz^1(M)$ is equidimensional.  Fix $P \in \min(\syz^1(M))$ and fix ring $A$ that is either regular or Gorenstein of the form $\dfrac{V[[x_1, \ldots, x_d]]}{(p^sx_1)}$ for some $s \geq 1$, where $V$ is a complete discrete valuation ring with maximal ideal $(p)$ over which $R$ is module finite.  By Lemma \ref{RegGorReduction}, $\syz^1(M)$ is also asymptotically Cohen--Macaulay over $A$ and so by Lemma \ref{torsionlemma} torsion-free over $A$.  Hence, $A \cap P$ must also be an associated prime of $A$, but then $d = \dim(A) = \dim(A/(A \cap P)) = \dim(R/P)$, as desired, because $A$ is Gorenstein.  It, therefore, follows from \cite[Theorem 2.4]{joint} that there exists a constant $c_{\syz^1(M)}$ such that $\ell(\syz^1(M)/I\syz^1(M)) \leq c_{\syz^1(M)} \cdot e_I(\syz^1(M))$ for every $m$-primary ideal $I$ of $R$.  Now because $M$ is asymptotically Cohen--Macaulay over $R$, we may fix $N \in \N$ such that for all parameter ideals $I \subseteq m^N$, we have $\dfrac{\ell(H^{d-1}(I;M))}{\ell(R/I)}<\dfrac{\varepsilon\cdot c_{\syz^1(M)}}{2}$.  Fix an arbitrary parameter ideal $\bar{J} \subseteq \mu^N$ and fix a $(d-1)$-generator lift $J$ of $\bar{J}$ to $R$ with $J \subseteq m^N$.  Note that for every $t \geq 1$, $J+(x^t)$ is a parameter ideal of $R$.  For each $t \geq N$, we observe

 \begin{align*}
\dfrac{\ell(H^{d-2}(\bar{J};M/xM))}{\ell(R/(J+(x)))} &= \dfrac{t \cdot \ell(H^{d-1}(J+(x);M))}{t \cdot \ell(R/(J+(x)))}\\
& \leq \dfrac{t \cdot \ell(H^{d-1}(J+(x);M))}{\ell(R/(J+(x^t)))}\\
& = \dfrac{t \cdot \ell(H^{d-1}(J+(x);M))}{\ell(H^{d-1}(J+(x^t);M))} \cdot \dfrac{\ell(H^{d-1}(J+(x^t);M))}{\ell(R/(J+(x^t)))}\\
&<\dfrac{t \cdot \ell(H^{d-1}(J+(x);M))}{\ell(H^{d-1}(J+(x^t);M))} \cdot \dfrac{\varepsilon \cdot c_{\syz^1(M)}}{2}.
\end{align*}
Hence, it suffices to show that there exists $t \geq N$ such that 
\[
\dfrac{\ell(H^{d-1}(J+(x);M))}{\ell(H^{d-1}(J+(x^t);M))/t} \leq \dfrac{2}{ c_{\syz^1(M)}}.
\]  
We will separately bound the numerator from above and the denominator from below.  From $0 \rightarrow \syz^1(M) \rightarrow R^h \rightarrow M \rightarrow 0$ and the fact that $R$ is Cohen--Macaulay, the long exact sequence of Koszul cohomology gives \[
0 \rightarrow H^{d-1}(J+(x);M) \rightarrow \dfrac{\syz^1(M)}{(J+(x))\syz^1(M)}
\] from which it follows that \begin{equation} \tag{$\clubsuit$}
\ell(H^{d-1}(J+(x);M)) \leq \nu_R(\syz^1(M)) \cdot \ell(R/(J+(x))).
\end{equation}
We now consider for each $t \geq 1$ the short exact sequence 
\[
0 \rightarrow \dfrac{H^{d-2}(J;M)}{x^t \cdot H^{d-2}(J;M)} \rightarrow H^{d-2}(J;M/(x^t)M) \rightarrow \Ann_{M/JM}(x^t) \rightarrow 0,
\]  
from which we see that \begin{equation} \tag{$\diamondsuit$}
\ell(H^{d-1}(J+(x^t);M)) = \ell(H^{d-2}(J;M/(x^t)M))  \geq \ell\left( \dfrac{H^{d-2}(J;M)}{x^t \cdot H^{d-2}(J;M)}\right).
\end{equation}  Because $H^{d-2}(J;M)$ is a one-dimensional $R/J$ module, there exists some $T \in \N$ such that for all $t \geq T$, 
\begin{equation} \tag{$\heartsuit$}
\ell\left( \dfrac{H^{d-2}(J;M)}{x^t \cdot H^{d-2}(J;M)}\right)/t \geq \dfrac{e_{(x)}(H^{d-2}(J;M))}{2}.
\end{equation}
Now $H^{d-2}(J;M) \cong H^{d-1}(J;\syz^1(M)) \cong \dfrac{\syz^1(M)}{J\syz^1(M)}$ because $R$ is Cohen--Macaulay and $J$ is generated by $(d-1)$ elements.  It follows that 
\begin{equation} \tag{$\spadesuit$}
e_{(x)}(H^{d-2}(J;M)) =e_{(x)}\dfrac{\syz^1(M)}{J\syz^1(M)} \geq e_{(J+(x))}(\syz^1(M)) \geq c_{\syz^1(M)} \cdot \ell(\syz^1(M)/(J+(x))\syz^1(M)).
\end{equation} 
Combining Equations $(\diamondsuit)$, $(\heartsuit)$, and $(\spadesuit)$, we obtain a lower bound on the denominator for all $t \geq T$, \begin{align*}
\dfrac{\ell(H^{d-1}(J+(x^t);M))}{t} &\geq \ell\left( \dfrac{H^{d-2}(J;M)}{x^t \cdot H^{d-2}(J;M)}\right)/t \\
&\geq \dfrac{e_{(x)}(H^{d-2}(J;M))}{2} \geq \dfrac{c_{\syz^1(M)}}{2} \cdot \ell\left(\dfrac{\syz^1(M)}{(J+(x))\syz^1(M)}\right).
\end{align*}

Combining this lower bound with the upper bound on the denominator from $(\clubsuit)$, we have for all $t \geq T$, \[
\dfrac{\ell(H^{d-1}(J+(x);M))}{\ell(H^{d-1}(J+(x^t);M))/t} \leq \dfrac{2 \cdot \ell(\syz^1(M)/(J+(x))\syz^1(M))}{c_{\syz^1(M)} \cdot \ell(\syz^1(M)/(J+(x))\syz^1(M))} = \dfrac{2}{c_{\syz^1(M)}}.
\] Finally, by choosing $N' > \max\{N, T\}$, we ensure that for all parameter ideals $I' \subseteq \mu^{N'}$, we have $\dfrac{\ell(H^{d-2}(I';M/xM))}{\ell((R/(x))/I'(R/(x)))} < \varepsilon$.
\end{proof}

 \begin{theorem}\label{forward}
 Let $(R,m,\kappa)$ be any local ring and $M$ a finitely-generated $R$-module with $\dim(M) = \dim(R)$.  If $M$ is asymptotically Cohen--Macaulay, then $M$ is Cohen--Macaulay on the punctured spectrum.  
 \end{theorem}
 
\begin{proof} Using Lemma \ref{completereduction}, we may assume that $R$ is complete.  By Lemma \ref{RegGorReduction}, we may assume that $R$ is either regular or Gorenstein of the form $V[[x_1, \ldots, x_{d}]]/(p^sx_1)$ with $d \geq 2$ and $s \geq 1$, where $V = (V, pV, \kappa)$ is a complete discrete valuation ring.    We assume that $M$ is asymptotically Cohen--Macaulay.  After replacing $M$ by $M/H^0_m(M)$, we may assume that $M$ is torsion-free by Lemma \ref{torsionlemma}.  Because all torsion-free modules of dimension 1 are Cohen--Macaulay, we may assume that $d-1>0$.  Fix a prime $P$ of $R$ of height $d-1$.  We aim to show that some system of parameters on $M_P$ is a regular sequence on $M_P$.  We claim that depth$M_P>0$.  If $M_P$ has depth 0, then $P$ is an associated prime of $M$ (equivalently, of $M_P$).  Then because $M$ is torsion-free, $P$ must be an associated prime of $R$, but then $P$ has height $0$, contradicting the assumption that $P$ has height $d-1>0$.  Therefore, we may fix $x_1, \ldots, x_{d-1}$ a system of parameters of $M_P$ with $x_1$ not a zerodivisor.  By \ref{modx}, $M/x_1M$ is asymptotically Cohen--Macaulay, and so $M/x_1M$ is Cohen--Macaulay on the punctured spectrum of $R/(x_1)$ by induction.  Hence $(M/x_1M)_P \cong M_P/x_1M_P$ is Cohen--Macaulay.  Now because $x_2, \ldots, x_{d-1}$ is a system of parameters on $R_P/x_1R_P$, it must also be a regular sequence.  It follows that $x_1, \ldots, x_{d-1}$ is a regular sequence on $M_P$.
 
\end{proof} 

      \begin{theorem} \label{topdim}
      If $M$ is a finitely generated module over the local ring $(R,m,\kappa)$ of with $\dim(M) = \dim(R) = d$, then the following three conditions are equivalent:
      \begin{enumerate}
      \item $M$ is asymptotically Cohen--Macaulay,
      \item $\sup\{\ell(H^i(f_1, \ldots, f_d;M))\mid \sqrt{f_1, \ldots, f_d} = m \mbox{, } i< d\}<\infty$,
      \item $M$ is equidimensional and Cohen--Macaulay on the punctured spectrum.
      \end{enumerate}
      \end{theorem}
      
      \begin{proof}
      It is obvious that (2) implies (1).  Using Lemma \ref{completereduction}, we may assume that $R$ and $M$ are complete, and using Lemma $\ref{RegGorReduction}$ we may assume that $R$ satisfies the hypotheses of Lemma \ref{forward}.  The implication (1) implies (3) follows from Lemma \ref{torsionlemma} together with Theorem \ref{forward}. Lastly, (3) implies (2) is the case $k = d$ of Theorem \ref{backward}.  
      \end{proof}
      
         Using the same technique, we give a result on rigidity of asymptotic depth under the assumption that $M$ is quasi-unmixed.
      
      \begin{corollary}\label{rigidity}
      Let $M$ be a finitely generated quasi-unmixed module over the local ring $(R,m,\kappa)$ with $\dim(M) = \dim(R) = d \geq 1$.  Fix $0 < i < d$.  If $M$ is $i$-effaceable, then $M$ is $j$-effaceable for every $0\leq j<i$.  
      \end{corollary}
      
      \begin{proof}
      Using Lemma \ref{completereduction}, we may assume that $R$ is complete and, by Lemma \ref{RegGorReduction}, Gorenstein.  Fix a module $M$ of minimal dimension giving a counterexample to the theorem.  We may replace $M$ by $M/H^0_m(M)$ without changin the problem and fix $x \in M$ a non-unit in $R$ and non-zerodivisor on $M$.  This proof will precede by describing appropriate modifications to Lemma \ref{modx}.  By taking syzygies as in Lemma \ref{modx}, we may assume that $M$ is $(d-1)$-effaceable and aim to show that it is asymptotically Cohen--Macaulay.  Because $M$ is equidimensional, so are all of its modules of syzygies, and so we have a constant $C_{\syz^1(M)}$ as described in Lemma \ref{modx}.  We may now follow the computations in Lemma \ref{modx} identically to conclude that $M/xM$ is $(d-2)$-effaceable.  By induction, asydepth$(M/xM) \geq d-1$, which is to say that $M$ is asymptotically Cohen--Macaulay.  But then $M/xM$ is Cohen--Macaulay on the punctured spectrum by Theorem \ref{topdim} and so $M$ is, following the argument of Theorem \ref{forward}.  Then $M$ is asymptotically Cohen--Macaulay by Theorem \ref{topdim}, as desired.  
      \end{proof}
      
      \begin{corollary} \label{forallk}
If $M$ is a finitely generated module over the local ring $(R,m,\kappa)$ of with $\dim(M) = \dim(R) = d$, then the following three conditions are equivalent for every $0 \leq k \leq d$: 
\begin{enumerate}
\item asydepth($M$) $\geq k$,
\item  $\sup\{\ell(H^i(f_1, \ldots, f_d;M))\mid \sqrt{f_1, \ldots, f_d} = m \mbox{, } i< k\}<\infty$,
\item $\ell(H^i_m(M)) < \infty$ for all $i<k$.
\end{enumerate} 
      \end{corollary} 
      
      \begin{proof}
      As above, it is clear that condition (2) implies condition (1).  To see that (1) implies (3), fix a local Gorenstein ring $(A,n)$ over which $(R,m)$ is module finite, and take \[
      0 \rightarrow \syz^1_A(M) \rightarrow A^\nu \rightarrow M \rightarrow 0 
      \] for $v = \nu_A(M)$.  We first establish the result over $A$.  If $k<d-1$, then for any parameter ideal $I$ of $A$, $H^k(I;M) \cong H^{k+1}(I;\syz_A^1(M))$, and $H^i_n(M) \cong H^{i+1}_n(\syz^1_A(M))$.  We may, then, without loss of generality assume $k = d-1$, which is the result of Theorem \ref{topdim}.  If condition (1) is satisfied over $R$, then it is certainly satisfied over $A$ because every system of parameters in $A$ is a system of parameters in $R$.  It then follows that $\ell(H^i_m(M)) = \ell(H^i_n(M))<\infty$ for all $i<k$, which is to say that condition (3) is satisfied over $R$.  Lastly, (3) implies (2) is Theorem \ref{backward}.  
      \end{proof}

      To finish the section, we give a direct proof that an asmptotically Cohen--Macaualy ring must always be quasi-unmixed using a result of  St\"{u}ckrad and Vogel on the relationship between multiplicity and colength when $M$ is not quasi-unmixed.  Their result is that, over a ring with an infinite residue field, if the set $\left \{\dfrac{e_I(M)}{\ell(M/IM)} \mid I \mbox{ parameter}\right\}$ is bounded away from $0$, then $M$ must be quasi-unmixed \cite[Theorem 1]{Vogel}.  
      
\begin{theorem} \label{equidimnecessary}
Let $(R, m, \kappa)$ be a local ring of dimension $d$.  If $R$ is asymptotically Cohen--Macaulay, then $R$ is quasi-unmixed.  
\end{theorem}

\begin{proof}
Using \ref{completereduction}, we may assume that $R$ is complete.  Replacing $R$ by $R(t)$, we assume that $|\kappa| = \infty$.  Assume $R$ is $(d-1)$-effaceable.  (By Corollary \ref{rigidity}, this assumption is equivalent to the assumption that $R$ is quasi-unmixed, so we are using the full power of our hypotheses.)  Then there exists some $N \in \N$ such that for all parameter ideals $I \subseteq m^N$, $\dfrac{\ell(H^{d-1}(I;R))}{\ell(R/I)}<1/2$.  Let $I = (f_1, \ldots, f_d)$ be any parameter ideal.  Then \begin{align*}
\dfrac{e_I(R)}{\ell(R/I)}  &= \dfrac{e_{(f_1^N, \ldots, f_d^N)}(R)}{N^d \cdot \ell(R/I)} \geq  \dfrac{e_{(f_1^N, \ldots, f_d^N)}(R)}{N^d \cdot \ell(R/(f_1^N, \ldots, f_d^N))} \\
&\geq \dfrac{1}{N^d}\left(1 - \dfrac{\ell(H^{d-1}(f_1^N, \ldots, f_d^N;R))}{\ell(R/(f_1^N, \ldots, f_d^N))} \right)\geq \frac{1}{2N^d},
\end{align*} and so the set $\left \{\dfrac{e_I(M)}{\ell(M/IM)} \mid I \mbox{ parameter}\right\}$ is bounded away from $0$.  It follows that $M$ is quasi-unmixed \cite[Theorem 1]{Vogel}.  
\end{proof}
      
      \section{Quotients of surprisingly large length} 

In this section, we study modules that arise as quotients of regular rings and Gorenstein rings of a specified form by ideals generated power series variables and powers of a generator of the maximal ideal of a discrete valuation ring.  In particular, we study lengths of quotients of these modules by special choices of parameter ideals of the ambient ring and compare these lengths to the lengths of quotients of the ambient rings by those same ideals.  We will typically separate the theorems depending on whether the ambient ring is regular or merely Gorenstein.  The proofs are very similar, and the reader interested in getting a flavor for the arguments might prefer to read only Proposition \ref{basecase}, Theorem \ref{d>2}, and Proposition \ref{d=2} and might prefer to think only about the equal characteristic case.  

\begin{proposition} \label{basecase}
If $R = \kappa[[x_1, \ldots, x_d]]$ with $d \geq 3$, then $M = (x_1, \ldots, x_{d-1})R$ is not asymptotically Cohen--Macaulay.  If $R = V[[x_1,x_2, \ldots, x_{d-1}]]$ where $V = (V,p,k)$ is a  discrete valuation ring, then neither of $M = (p, x_1, \ldots, x_{d-2})$ nor $M = (x_1, \ldots, x_{d-1})$ is asymptotically Cohen--Macaulay.  In all cases, $N = R/M$ is also not asymptotically Cohen--Macaulay.  In particular, $N$ is not $1$-effaceable and $M$ is not $2$-effaceable.
\end{proposition}

\begin{proof} Before giving the general proof, we will compute in detail the case of $d = 3$ with $R = \kappa[[x,y,z]]$ as a guiding example: Let \[
I_n = (z^{n^4}-z^nx^n,y^n-z^nx,x^{n+1}-xz^{n^4-n}+yz^n).
\]  It is easy to see that $\ell(R/(I_n+(x,y))) = n^4$, and the computation below will show that $\ell(R/I_n) \leq (n^4+2n)+(2n+1)^2(3n)$.  We claim that $x^{2n+1}, y^{2n+1}, yz^{2n}, xz^{3n}$, and $z^{n^4+2n}$ are elements of $I_n$, and so the elements $z^i$ with $i<n^4+2n$ and $z^iy^jz^k$ with $i<3n$; $j$, $k<2n+1$ span the quotient $R/I_n$ as a $k$-vector space (though they will not in general form a basis).  

The claimed inclusions can be seen in the following identities:

\[
yz^{2n} = z^n(x^{n+1}-xz^{n^4-n}+yz^n)+x(z^{n^4}-z^nx^n)
\]
\[
xz^{3n} = -z^{2n}(y^n-z^nx)+y^nz^{2n} 
\]
\[
x^{n+1}+yz^n = (x^{n+1}-xz^{n^4-n}+yz^n)+z^{n^4-4n}(xz^{3n})
\]
\[
z^nx^{n+1} = z^n(x^{n+1}+yz^n)-yz^{2n}
\]
\[
x^ny^n = x^n(y^n-z^nx)+z^nx^{n+1}
\]
\[
y^{2n+1} = y^{n+1}(y^n-z^nx)+y^nx(yz^n+x^{n+1})-x^2(x^ny^n)
\]
\[
x^nyz^n = -y(z^{n^4}-z^nx^n)+yz^{n^4}
\]
\[
x^{2n+1} = x^n(x^{n+1}-xz^{n^4-n}+yz^n)+x^{n+1}z^{n^4-n}-x^nyz^n
\]
\[
z^{n^4+2n} = z^{2n}(z^{n^4}-z^nx^n)-x^nz^{3n}
\]

\bigskip

Elements on the left can be seen to be in $I_n$ as they are put in terms of elements already known to be in $I_n$ on the righthand side of each equation.

\bigskip

Because $R$ is regular, $H^i(I_n;R) = 0$ for $i<3$, and so the long exact sequence of Koszul cohomology yields \[
0 \rightarrow H^1(I_n;\kappa[[z]]) \rightarrow H^2(I_n;M) \rightarrow 0.
\]  Because $z^{n^4}$ is not a zerodivisor on $\kappa[[z]]$ and the image of $I_n$ in  $\kappa[[z]]$ is $(z^{n^4},0,0)$, we have that  \[
H^1(I_n;N) \cong H^0(0,0; \kappa[[z]]/(z^{n^4})) \cong R/(I_n+(x,y)),\] whose length is $n^4$.  From the long exact sequence above, $n^4$ must also be the length of $H^2(I_n;M)$.  Now because \[\displaystyle \lim_{n \to\infty}\dfrac{\ell(H^2(I_n;M))}{\ell(R/I_n)} \geq  \displaystyle \lim_{n \to\infty}  \dfrac{n^4}{n^4+2n+(2n+1)^2(3n)}  = 1> 0,\] $M$ is not 2-effaceable and so in particular is not asymptotically Cohen--Macaulay.

\bigskip 

We now consider all dimensions $d \geq 3$.  Let $R = \kappa[[x,y,z, v_1, \ldots, v_{d-3}]]$, $M = (x,y,v_1, \ldots, v_{d-3})R$, or $R = V[[y,z, v_1, \ldots, v_{d-3}]]$ and $M = (p,y, v_1, \ldots, v_{d-3})$, in which case we will denote $p$ by $x$ below, or $R = V[[x,y, v_1, \ldots, v_{d-3}]]$ and $M =(x,y, v_1, \ldots, v_{d-3})$ in which case we will denote  $p$ by $z$ below.  In all cases, take $N = R/M$.  From the short exact sequence $0 \rightarrow M \rightarrow R \rightarrow N \rightarrow 0$ and the fact that $R$ is regular, we know $H^i(I_n;M) \cong H^{i-1}(I_n; \kappa[[z]])$ or $H^i(I_n;M) \cong H^{i-1}(I_n;V)$ for all $i \leq d-1$.  We aim to show that $N$ is not 1-effaceable and so that $M$ is not 2-effaceable.  

We define $I_n = (f_1, \ldots, f_d)$ where $f_1 = z^t-z^nx^n$, $f_2 = x^{n+1}-xz^{t-n}+yz^n$, \[
f_3 = y^n+v_1z^n-v_2z^n+\cdots+(-1)^{i+1}v_iz^n+\cdots+(-1)^{d-2}v_{d-3}z^n+(-1)^{d-3}xz^n,
\] and \[
f_{i+3} = v_i^n+v_iz^{t-n}-(v_{i+1}-v_{i+2}+\cdots+(-1)^{d+i}(v_{d-3})+(-1)^{d+i+1}x)^n+v_iz^n-v_ix^n
\] for $1 \leq i \leq d-3$ and some $t \in \N$.  We will eventually choose $t$ to be very large relative to $n$.

As in the 3-dimensional case, we may use the first two equation to show $yz^{2n} \in I_n$.   We then find \[
\left(\sum_{i=1}^{d-3}(-1)^{i+1}v_iz^{3n}\right)+(-1)^{d-1}(xz^{3n}) =z^{2n}f_3-yz^{2n}  \in I_n.
\]  We will now show by induction on $i$ that $v_iz^{3n^{i+1}+n^i+\cdots+n}$ and \[
\left(\sum_{j=i+1}^{d-3}(-1)^{j}v_jz^{3n^{i+2}+n^{i+1}+\cdots+n^2+n}\right)+(-1)^{d}(xz^{3n^{i+2}+n^{i+1}+\cdots+n^2+n})
\] are elements of $I_n$ for all $1 \leq i \leq d-3$.  If $i=1$, then, using that modulo $I_n$ \[
v_1z^{3n} \equiv \left(\sum_{j=2}^{d-3}(-1)^{j}v_jz^{3n}\right)+(-1)^{d}(xz^{3n})\]
 implies that modulo $I_n$, \[
 (v_1z^{3n})^n \equiv \left(\left(\sum_{j=2}^{d-3}(-1)^{j}v_iz^{3n}\right)+(-1)^{d}(xz^{3n})\right)^n.
 \] We compute \[
-v_1z^{3n^2+n} = v_1z^{3n^2-n}(z^t-z^nx^n)-z^{3n^2}f_4+(v_1z^{3n})^n-\left(\left(\sum_{j=2}^{d-3}(-1)^{j}v_iz^{3n}\right)+(-1)^{d}(xz^{3n})\right)^n,
\]  where the right-hand side consists of elements known to be in $I_n$.  It follows that \[
\left(\sum_{j=2}^{d-3}(-1)^{j}v_jz^{3n^2+n}\right)+(-1)^{d}(xz^{3n^2+n}) = -z^{3n^2}f_3+v_1z^{3n^2+n}+y^nz^{3n^2} \in I_n.
\]

For the inductive step, we compute \begin{align*}
&-v_{i+1}z^{3n^{i+2}+n^{i+1}+\cdots+n^2+n} \\
&= v_{i+1}z^{3n^{i+2}+n^{i+1}+\cdots+n^2-n}(z^t-z^nx^n)-z^{3n^{i+2}+n^{i+1}+\cdots+n^2}f_{i+4}+(v_{i+1}z^{3n^{i+1}+n^{i}+\cdots+n^2+n})^n\\
&-\left(\left(\sum_{j=i+2}^{d-3}v_jz^{3n^{i+1}+n^{i}+\cdots+n^2+n}\right)+(-1)^{d-1}xz^{3n^{i+1}+n^{i}+\cdots+n^2+n}\right)^n \in I_n,
\end{align*} from which it follows that \begin{align*}
&\left(\sum_{j=i+2}^{d-3}(-1)^{j}v_jz^{3n^{i+2}+n^{i+1}+\cdots+n^2+n}\right)+(-1)^{d}(xz^{3n^{i+2}+n^{i+1}+\cdots+n^2+n})\\
& = -z^{3n^{i+2}+n^{i+1}+\cdots+n^2}f_3 + \left(\sum_{j=1}^{i+1} (-1)^{j+1} v_jz^{3n^{i+2}+n^{i+1}+\cdots+n^2+n}\right)+y^nz^{3n^{i+2}+n^{i+1}+\cdots+n^2} \in I_n.
\end{align*}

In particular, we have that $v_{d-3}z^{3n^{d-2}+n^{d-3}+\cdots+n^2+n} \in I_n$ and that
\[
(-1)^{d-3}v_{d-3}z^{3n^{d-2}+n^{d-3}+\cdots+n^2+n} +(-1)^d xz^{3n^{d-2}+n^{d-3}+\cdots+n^2+n} \in I_n,
\]  and so $xz^{3n^{d-2}+n^{d-3}+\cdots+n^2+n} \in I_n$ as well.

  In particular, $xz^{3n^{d-2}+n^{d-3}+\cdots+n^2+n}, yz^{3n^{d-2}+n^{d-3}+\cdots+n^2+n}$ and $v_iz^{3n^{d-2}+n^{d-3}+\cdots+n^2+n}$ are elements of $I_n$ for all $1 \leq i \leq d-3$.  Now using $f_2 = x^{n+1}-xz^{t-n}+yz^n$, we notice that modulo $I_n$ \[(x^{n+1})^{3n^{d-3}+n^{d-4}+\cdots+n^2+n}  \equiv (xz^{t-n}-yz^n)^{3n^{d-3}+n^{d-4}+\cdots+n^2+n} \equiv 0
\] and then that \begin{align*}
&z^{t+(3n^{d-2}+n^{d-3}+\cdots+n^2)} \\
&= z^{(3n^{d-2}+n^{d-3}+\cdots+n^2)}(z^t-z^nx^n)+(x^{n-1})(xz^{3n^{d-2}+n^{d-3}+\cdots+n^2+n}) \in I_n.  
\end{align*} We may also use $f_3 = y^n+v_1z^n-v_2z^n+\cdots+(-1)^{i+1}v_iz^n+\cdots+(-1)^{d-2}v_{d-3}z^n+(-1)^{d-3}xz^n$ to see that modulo $I_n$ \[
(y^n)^{3n^{d-3}+n^{d-4}+\cdots+n+1}  \equiv \left(\left(\sum_{j=1}^{d-3}(-1)^{j}v_iz^n\right)+(-1)^{d}xz^n\right)^{3n^{d-3}+n^{d-4}+\cdots+n+1} \equiv 0.
\]  We will now show by induction on $k$ that for $0 \leq k \leq d-4$\[
v_{d-3-k}^{(n+1)(3n^{d-3}+n^{d-4}+\cdots+n^2+n)+(3n^{d-2}+n^{d-3}+\cdots+n^2+n)} \in I_n.\]
  Set $N_0 = (n+1)(3n^{d-3}+n^{d-4}+\cdots+n^2+n)+(3n^{d-2}+n^{d-3}+\cdots+n^2+n)$.  We choose this $N_0$ because $x^{(n+1)(3n^{d-3}+n^{d-4}+\cdots+n^2+n)} \in I_n$ and $v_iz^{(3n^{d-2}+n^{d-3}+\cdots+n^2+n)} \in I_n$ for each $1 \leq i \leq d-3$, and so any product of $N_0$ elements, each of which is divisible either by $x$ or by $z^n$, is an element of $I_n$ provided that whenever it is not divisible by $x^{(n+1)(3n^{d-3}+n^{d-4}+\cdots+n^2+n)}$ it is divisible by some $v_i $ for $1 \leq i \leq d-3$.  When $k=0$, we use $f_d  = v_{d-3}^n+v_{d-3}z^{t-n}+x^n+v_{d-3}z^n-v_{d-3}x^n$ to see that modulo $I_n$\[
  v_{d-3}^{N_0} \equiv (-v_{d-3}z^{t-n}-x^n-v_{d-3}z^n+v_{d-3}x^n)^{N_0} \equiv 0.
  \]  For the inductive step, we set $N_k = 2N_{k-1}$ and use \begin{align*}
  &f_{d-k} = v_{d-3-k}^n+v_{d-3-k}z^{t-n}\\
 & -(v_{d-3-k+1}-v_{d-3-k+2}+\cdots+(-1)^{d+d-3-k}(v_{d-3})+(-1)^{d+d-3-k+1}x)^n+v_{d-3-k}z^n-v_{d-3-k}x^n.
  \end{align*} Modulo $I_n$
\begin{align*}
&v_{d-3-k}^{N_k} \equiv \left(-v_{d-3-k}z^{t-n}+ \left(\left(\sum_{j = d-k-2}^{d-3}(-1)^{j+d+k} v_j \right)+(-1)^kx \right)^n-v_{d-k-3}z^n+v_{d-k-3}x^n \right)^{N_k} \equiv 0,
\end{align*} by the pigeon-hole principle because the product of any $N_{k-1}$ elements, each of which is a multiple of $z^{t-n}, v_{d-3-k}z^n, x$, or $v_j$ for $d-1-k \leq j \leq d-3$, is in $I_n$ and $v_{d-2-k}^{N_{k-1}} \in I_n$.  Using $k = d-4$ for $v_1 = v_{d-3-(d-4)}$, we have $N_{d-4} = 2^{d-4}(n+1)(3n^{d-3}+n^{d-4}+\cdots+n^2+n)+(3n^{d-2}+n^{d-3}+\cdots+n^2+n)$, and so $v_1^{2^{d-4}(n+1)(3n^{d-3}+n^{d-4}+\cdots+n^2+n)+(3n^{d-2}+n^{d-3}+\cdots+n^2+n)} \in I_n$, and, more generally, $v_i^{2^{d-4}(n+1)(3n^{d-3}+n^{d-4}+\cdots+n^2+n)+(3n^{d-2}+n^{d-3}+\cdots+n^2+n)} \in I_n$ for all $1 \leq i \leq d-3$.  

The attentive reader will notice that efforts have not been made to keep minimal the degrees of polynomials appearing in the above expressions.  It is now clear that $R/I_n$ is spanned by $xz^j$, $yz^j$, and $v_iz^j$ for $j < {(3n^{d-2}+n^{d-3}+\cdots+n^2+n)}$ and $1 \leq i \leq d-3$ together with $z^j$ for $j<{t+(3n^{d-2}+n^{d-3}+\cdots+n^2)}$ and $x^\alpha_{x}y^{\alpha_y}v_1^{\alpha_1} \cdots v_{d-3}^{\alpha_{d-3}}$ with $\alpha_y<n(3n^{d-3}+n^{d-4}+\cdots+n+1)$ and $\alpha_x<(n+1)(3n^{d-3}+n^{d-4}+\cdots+n^2+n)$, and $\alpha_i<2^{d-4}(n+1)(3n^{d-3}+n^{d-4}+\cdots+n^2+n)+(3n^{d-2}+n^{d-3}+\cdots+n^2+n)$.  Hence, for sufficiently large $n$, \[
\ell(R/I_n) \leq t+(3n^{d-2}+n^{d-3}+\cdots+n^2)+(4(2n)^{d-2})^d.
\] The term $t+(3n^{d-2}+n^{d-3}+\cdots+n^2)$ counts powers of $z$, and $(4(2n)^{d-2})^d$ bounds ways to pick an allowable monomial that is not a power of $z$ for $n>>0$.  As in the 3-dimensional case, $H^2(I_n;M) \cong H^1(I_n;N) \cong H^0(0,0; \kappa[[z]]/(z^{t})) \cong R/(I_n+(x,y, v_1, \ldots, v_{d-3}))$, whose length is easily seen to be $t$.  Because $R/I_n$ surjects onto $R/(I_n+(x,y, v_1, \ldots, v_{d-3}))$, we have $\ell(H^2(I_n;M)) \leq \ell(R/I_n)$.  Any choice of $t$ much larger than $(3n^{d-2}+n^{d-3}+\cdots+n^2)+(4(2n)^{d-2})^d$,
  for example $t = n^{d^2}$, will give \[
1 \geq \lim_{n \rightarrow \infty} \dfrac{\ell(H^2(I_n;M))}{\ell(R/I_n)} \geq \lim_{n \rightarrow \infty} \dfrac{t}{t+(3n^{d-2}+n^{d-3}+\cdots+n^2)+(4(2n)^{d-2})^d} = 1.
\] This computation demonstrates that neither $N$ nor $M$ is asymptotically Cohen--Macaulay.  In particular, $N$ is not 1-effaceable, and $M$ is not 2-effaceable.  
\end{proof}

\begin{corollary}\label{s>1}
If $R = V[[x_1, \ldots, x_{d-1}]]$ where $V = (V, pV, \kappa)$ is a  discrete valuation ring, $M = (p^k, x_1, \ldots, x_{d-2})R$, and $N = R/M \cong \dfrac{V}{(p^k)}[[x_{d-1}]]$ with $d \geq 2$ and $k \geq 1$, then $N$ is not $1$-effaceable and $M$ is not $2$-effaceable.  
\end{corollary}

\begin{proof}
We proceed by induction on $k$.  The case $k = 1$ is Proposition \ref{basecase}.  For $k>1$, we use the short exact sequence $0 \rightarrow p^{k-1}N \rightarrow N \rightarrow \dfrac{V}{(p^{k-1})}[[x_{d-1}]] \rightarrow 0$ and the fact that $\dfrac{V}{(p^{k-1})}[[x_{d}]]$ has depth 1 to obtain an injection $H^1(I;p^{k-1}N) \hookrightarrow H^1(I;N)$ from the long exact sequence of Koszul cohomology for every parameter ideal $I$.  The result now follows from the isomorphisms $p^{k-1}N \cong \kappa[[x_{d-1}]]$ and $H^1(I;N) \cong H^2(I;M)$.   
\end{proof}

We now prove the results of Proposition \ref{basecase} and Corollary \ref{s>1} over a the ring that is a certain type of Gorenstein ring rather than regular. 

\begin{proposition} \label{basecasemixed}
Let $R = \dfrac{V[[x_1, \ldots, x_d]]}{(p^sx_1)}$ with $d \geq 3$ where $V = (V, pV, \kappa)$ is a  discrete valuation ring and $s \geq 1$. Then neither $M = (p, x_1, \ldots, x_{d-1})$ nor $M = (x_1, \ldots, x_d)$ nor $M = (p, x_2, \ldots, x_d)$ is asymptotically Cohen--Macaulay.  In all cases, $N = R/M$ is also not asymptotically Cohen--Macaulay.  In particular, $N$ is not $1$-effaceable and $M$ is not $2$-effaceable.
\end{proposition} 

\begin{proof}

As in Proposition \ref{basecase}, we consider the short exact sequence $0 \rightarrow M \rightarrow R \rightarrow N \rightarrow 0$ and, using that $R$ is Cohen--Macaulay and dimension at least 3, note that $H^2(I;M) \cong H^1(I;N)$ for every parameter ideal $I$.  We will again choose parameter ideals $I_n$ so that $H^1(I_n;N) \cong R/(I_n+P)$ for some prime ideal $P$ of $R$ with $R/P \cong V$ or $R/P \cong \kappa[[y]]$ for some indeterminate $y$.  

Set $I_n = (f_1, \ldots, f_d)$ where $f_1 = z^t-z^n(x-p)^n$, $f_2 = (x-p)^{n+1}-(x-p)z^{t-n}+yz^n$, \[
f_3 = y^n+v_1z^n-v_2z^n+\cdots+(-1)^{i+1}v_iz^n+\cdots+(-1)^{d-2}v_{d-3}z^n+(-1)^{d-3}(x-p)z^n,\]
 and \[
 f_{i+3} = v_i^n+v_iz^{t-n}-(v_{i+1}-v_{i+2}+\cdots+(-1)^{d+i}v_{d-3}+(-1)^{d+i+1}(x-p))^n+v_iz^n-v_ix^n
 \] for $1 \leq i \leq d-3$ and some $t \in \N$.  Notice that these are the same generators used in Propostion \ref{basecase} except that $x-p$ replaces $x$ everywhere it appears.  

%

We now follow identically the computations of Proposition \ref{basecase}, borrowing its notation, with $x$ replaced by $x-p$ to see that that $((x-p)^{n+1})^{3n^{d-3}+n^{d-4}+\cdots+n^2+n}$, $z^{t+(3n^{d-2}+n^{d-3}+\cdots+n^2)}$, \newline $(y^n)^{3n^{d-3}+n^{d-4}+\cdots+n+1}$, and $v_i^{2^{d-4}(n+1)(3n^{d-3}+n^{d-4}+\cdots+n^2+n)+(3n^{d-2}+n^{d-3}+\cdots+n^2+n)}$ for $1 \leq i \leq d-3$ are elements of $I_n$.  Multiplying the first element by $x$ or by $p$, we have $(x^{n+2})^{3n^{d-3}+n^{d-4}+\cdots+n^2+n}$ and $(p^{n+2})^{3n^{d-3}+n^{d-4}+\cdots+n^2+n}$ elements of $I_n$ as well.  Similarly, $(x-p)z^{3n^{d-2}+n^{d-3}+\cdots+n^2+n}$, \newline $yz^{3n^{d-2}+n^{d-3}+\cdots+n^2+n}$ and $v_iz^{3n^{d-2}+n^{d-3}+\cdots+n^2+n}$ are elements of $I_n$ for $1 \leq i \leq d-3$.  Again, we multiply the first of these last three elements by $x$ or by $p$ to find $x^2z^{3n^{d-2}+n^{d-3}+\cdots+n^2+n}$ and $p^2z^{3n^{d-2}+n^{d-3}+\cdots+n^2+n}$ as elements of $I_n$ as well.  It is this final modification that prevents us from establishing a limit of 1 as we were able to in Proposition \ref{basecase}.  
 
We now count elements spanning $R/I_n$ as in Proposition \ref{basecase}, with the different that we must also allow $xz^j$ and $pz^j$ for $j$ small enough that $z^j \notin I_n$.  We note that $R/I_n$ is spanned by $x^2z^j$, $p^2z^j$, $yz^j$, and $v_iz^j$ for $j < {3n^{d-2}+n^{d-3}+\cdots+n^2+n}$ and $1 \leq i \leq d-3$ together with $z^j$, $pz^j$, and $xz^j$ for $j<{t+(3n^{d-2}+n^{d-3}+\cdots+n^2)}$ and $x^{\alpha_{xp}}y^{\alpha_y}v_1^{\alpha_1} \cdots v_{d-3}^{\alpha_{d-3}}$ and $p^{\alpha_{xp}}y^{\alpha_y}v_1^{\alpha_1} \cdots v_{d-3}^{\alpha_{d-3}}$ with $\alpha_y<n(3n^{d-3}+n^{d-4}+\cdots+n+1)$ and $\alpha_{xp}<({n+2})({3n^{d-3}+n^{d-4}+\cdots+n^2+n})$, and $\alpha_i<2^{d-4}(n+1)(3n^{d-3}+n^{d-4}+\cdots+n^2+n)+(3n^{d-2}+n^{d-3}+\cdots+n^2+n)$.  Therefore, \[
\ell(R/I_n) \leq 3({t+(3n^{d-2}+n^{d-3}+\cdots+n^2)})+(8(2n)^{d-2})^d,
\] and so, if $\displaystyle \lim_{n \rightarrow \infty} \dfrac{\ell(N/I_nN)}{\ell(R/I_n)}$ exists, then \[
1 \geq \lim_{n \rightarrow \infty} \dfrac{\ell(N/I_nN)}{\ell(R/I_n)}\geq \lim_{n \rightarrow \infty}\dfrac{t}{\ell(R/I_n) \leq 3({t+(3n^{d-2}+n^{d-3}+\cdots+n^2)})+(8(2n)^{d-2})^d} =1/3
\]
for every $t>>n^{d^2}$.  Because $\dfrac{\ell(N/I_nN)}{\ell(R/I_n)} = \dfrac{\ell(H^1(I_n;N))}{\ell(R/I_n)} = \dfrac{\ell(H^2(I_n;M))}{\ell(R/I_n)}$, this completes the proof.
\end{proof}

\begin{corollary} \label{basecasemixednonprime}
Let $R = \dfrac{V[[x_1, \ldots, x_d]]}{(p^sx_1)}$ with $d \geq 3$ where $V = (V, pV, \kappa)$ is a  discrete valuation ring and $s \geq 1$. Then neither $M = (p^k, x_1, \ldots, x_{d-1})R$ nor $M = (p^k, x_2, \ldots, x_d)R$ for any $k \geq 1$ is asymptotically Cohen--Macaulay.  In all cases, $N = R/M$ is also not asymptotically Cohen--Macaulay.  In particular, $N$ is not $1$-effaceable and $M$ is not $2$-effaceable.
\end{corollary}

\begin{proof}
We proceed by induction on $k$.  The case $k = 1$ is Proposition \ref{basecasemixed}.  For $k >1$, if $M = (p^k, x_1, \ldots, x_{d-1})R$ or if $M = (p^k, x_2, \ldots, x_{d})R$ and $k \leq s$, we may follow identically the proof of Corollary \ref{s>1}.  If $M = (p^k, x_2, \ldots, x_{d})R$ and $k>s$, observe that the short exact sequence $0 \rightarrow x_1N \rightarrow N \rightarrow N/x_1N \rightarrow 0$ gives rise to $H^0(I; V/(p^k)) \rightarrow H^1(I;x_1N) \rightarrow H^1(I;N)$ for each parameter ideal $I$ so that $\ell(H^1(I;N)) \geq \ell(H^1(I;x_1N))-k$.  Noting that $x_1N \cong V[[x_1]]/(p^s)$, which is the case of $s=k$, completes the proof.  
\end{proof}

\bigskip

\begin{theorem} \label{d>2} Suppose $R = \kappa[[x_1, \ldots, x_d]]$ with $d \geq 3$ and $M = (x_1, \ldots, x_{d-h})R$ or that $R = V[[x_1, \ldots, x_{d-1}]]$ and $M = (x_1, \ldots, x_{d-h})R$ or $M = (p, x_2, \ldots, x_{d-h})R$ or some $1\leq h<d-1$.  Then neither $M$ nor $N = R/M$ is asymptotically Cohen--Macaulay.  In particular, $M$ is not $h+1$-effaceable, and $N$ is not $h$-effaceable.  
\end{theorem}

\begin{proof} We will proceed by induction on $h$.  The base case $h=1$ is Proposition \ref{basecase}.  For the inductive step, we consider the following short exact sequences\[
0 \rightarrow (x_1, \ldots, x_{d-(h+1)})R \rightarrow (x_1, \ldots, x_{d-h})R \rightarrow (x_{d-h})\kappa[[x_{d-h}, \ldots, x_d]] \rightarrow 0
\] or 
\[0 \rightarrow (x_1, \ldots, x_{d-(h+1)})R \rightarrow (x_1, \ldots, x_{d-h})R \rightarrow (x_{d-h})V[[x_{d-h}, \ldots, x_{d-1}]] \rightarrow 0\] or 
\[0 \rightarrow (x_2, \ldots, x_{d-h})R \rightarrow (p,x_2, \ldots, x_{d-h})R \rightarrow (p)V[[x_1, x_{d-h+1}, \ldots, x_{d-1}]] \rightarrow 0
\] 
Call the module appearing as the cokernel in each of these sequences $aC$ where $a = x_{d-h}$ or $a = p$ and $C = \kappa[[x_{d-h}, \ldots, x_d]]$, $C = V[[x_{d-h}, \ldots, x_{d-1}]]$, or $C = V[[x_1, x_{d-h+1}, \ldots, x_{d-1}]]$.  Call the middle term of each sequence $D$ and the left-hand term $E$.  We observe that the depth of $E$ is $h+2$, and the depths of $D$ and $aC$ are $h+1$.  We, therefore, have \[
0 \rightarrow H^{h+1}(I_n;D) \rightarrow H^{h+1}(I_n;aC)
\] which shows that $aC\cong C$ is not $(h+1)$-effaceable since $D$ is not by the inductive hypothesis.  

We now consider another short exact sequence: \[
0 \rightarrow E \rightarrow R \rightarrow C \rightarrow 0 
\]  for each $C$ and $E$ defined above.  Because $R$ has depth $d \geq h+2$ by assumption, the long exact sequence of Koszul homology yields \[
0 \rightarrow H^{h+1}(I_n; C) \rightarrow H^{h+2}(I_n;E),
\] and so $E$ is not $(h+2)$-effaceable because $C$ is not $(h+1)$-effaceable, completing the induction.  Because $R$ is Cohen--Macaulay, the short exact sequence $0 \rightarrow M \rightarrow R \rightarrow N \rightarrow 0$ gives immediately that whenever $M$ is not $h+1$-effaceable $N$ is not $h$-effaceable for $1 \leq h<d-1$.

\end{proof}

\begin{corollary} \label{d>2notprime}
If $R = V[[x_1, \ldots, x_{d-1}]]$ and $M = (p^k, x_2, \ldots, x_{d-h})$ for some $k>1$ and $1\leq h<d-1$, then neither $M$ nor $N = R/M$ is asymptotically Cohen--Macaulay.  In particular, $M$ is not $h+1$-effaceable and $N$ is not $h$-effaceable.  
\end{corollary}
\begin{proof}
Using induction on $h$, the base case is Proposition \ref{basecasemixed}.  We now follow the proof of Theorem \ref{d>2}, using the short exact sequence \[0 \rightarrow (x_2, \ldots, x_{d-h})R \rightarrow (p^k, x_2, \ldots, x_{d-h})R \rightarrow (p^k)V[[x_{d-h}, \ldots, x_{d}]] \rightarrow 0
\] or 
\begin{align*}0 \rightarrow (p^k, x_2, &\ldots, x_{d-h-1})R \rightarrow \\
&(p^k, x_2, \ldots, x_{d-h})R \rightarrow (x_{d-h})\dfrac{V}{(p^k)}[[x_{d-h}, \ldots, x_{d-1}]] \rightarrow 0. \end{align*}
\end{proof}

\begin{theorem} \label{d>2mixed}
Let $R = \dfrac{V[[x_1, \ldots, x_d]]}{(p^sx_1)}$ with $d \geq 3$ where $V = (V, pV, \kappa)$ is a discrete valuation ring and $s \geq 1$.  Let $M = (p, x_1, \ldots, x_{d-h})R$ or $M = (p, x_2, \ldots, x_{d-h+1})R$ or $M = (x_1, \ldots, x_{d-h+1})R$ for some $1\leq h < d-1$.  Then neither $M$ nor $N = R/M$ is asymptotically Cohen--Macaulay.  In particular, $M$ is not $h+1$-effaceable, and $N$ is not $h$-effaceable.
\end{theorem}

\begin{proof}
The base case $h=1$ is Proposition \ref{basecasemixed}.  For the inductive step, we follow the proof of Theorem \ref{d>2} using the short exact sequence \[
0 \rightarrow (p, x_1, \ldots, x_{d-h-1})R \rightarrow (p, x_1, \ldots, x_{d-h})R \rightarrow (x_{d-h})\kappa[[x_{d-h}, \ldots, x_d]] \rightarrow 0, 
\] \[
0 \rightarrow (x_1, \ldots, x_{d-h})R \rightarrow (p, x_1, \ldots, x_{d-h})R \rightarrow (p)V[[x_{d-h+1}, \ldots, x_d]] \rightarrow 0, 
\] \[
0 \rightarrow (p, x_2, \ldots, x_{d-h})R \rightarrow (p, x_2, \ldots, x_{d-h+1})R \rightarrow (x_{d-h+1})\kappa[[x_1, x_{d-h+1}, \ldots, x_d]] \rightarrow 0, 
\] or \[
0 \rightarrow (x_1, \ldots, x_{d-h})R \rightarrow (x_1, \ldots, x_{d-h+1})R \rightarrow (x_{d-h+1})V[[x_{d-h+1}, \ldots, x_d]] \rightarrow 0. 
\]
\end{proof}

\bigskip

\begin{corollary} \label{d>2mixednotprime}
Let $R = \dfrac{V[[x_1, \ldots, x_d]]}{(p^sx_1)}$ with $d \geq 3$ where $V = (V, pV, \kappa)$ is a discrete valuation ring and $s \geq 1$.  Let $M = (p^k, x_1, \ldots, x_{d-h})R$ or $M = (p^k, x_2, \ldots, x_{d-h+1})R$ for some $1\leq h < d-1$ and $k \geq 1$.  Then neither $M$ nor $N = R/M$ is asymptotically Cohen--Macaulay.  In particular, $M$ is not $h+1$-effaceable, and $N$ is not $h$-effaceable.
\end{corollary}

\begin{proof}
We follow the proof of Corollary \ref{basecasemixednonprime} using Theorem \ref{d>2mixed} for the base case $k = 1$.  
\end{proof}

\begin{proposition} \label{d=2} Let $R = \kappa[[x_1, \ldots, x_d]]$ or $R = V[[x_1, \ldots, x_{d-1}]]$ where $V  = (V, pV, \kappa)$ is a  discrete valuation ring and \[
N = \dfrac{\kappa[[x_1, \ldots, x_d]]}{(x_d)} \cong \kappa[[x_1, \cdots, x_{d-1}]],\] \[N = V[[x_1, \ldots, x_{d-1}]]/(p^k) \cong \dfrac{V}{(p^k)}[[x_1, \ldots, x_{d-1}]]\mbox{ for $k \geq 1$},\] or \[N = V[[x_1, \ldots, x_{d-1}]]/(x_{d-1}) \cong V[[x_1, \ldots, x_{d-2}]].\] Then $N$ is not $(d-1)$-effaceable for $d \geq 2$. \end{proposition}

\begin{proof} Call the generator of the principle ideal by which we are quotienting $R$ by to obtain $N$ in each case $a$.  We note that the techniques of the previous lemmas do not work here because $aR \cong R$.  Using the short exact sequence $0 \rightarrow R \xrightarrow{a} R \rightarrow N \rightarrow 0$, we get the long exact sequence \[
0 \rightarrow H^{d-1}(I_n; N) \rightarrow R/I_n \rightarrow R/I_n \rightarrow N/I_nN \rightarrow 0
\] from which we see that $\ell(H^{d-1}(I_n; N)) = \ell(N/I_nN)$.  It is therefore sufficient to give a family of ideals $I_n$ such that $\displaystyle \lim_{n \to\infty} \dfrac{\ell(R/((a)+I_n))}{\ell(R/I_n)} \neq 0$.  We have previously discussed such families for $d \geq 3$ since \[
\ell(R/((a, x_1, \ldots, x_{d-2})+I_n)) \leq \ell(R/((a)+ I_n)).\]
  For $d = 2$, we may use the family $I_n = (x_1^{n+1}-x_1 x_2^n, x_2^{n^3}+x_1^n)$ with $x_2 = p$ when $R = V[[x_1]]$.  It is clear that $\ell(R/((x_1)+I_n)) = n^3$, and because neither $x_1$ nor $x_1^n-x_2^n$ is a zerodivisor on $R/(x_2^{n^3}+x_1^n)$, we may compute $\ell(R/I_n) = \ell(R/(x_1, x_2^{n^3}))+\ell(R/(x_1^n-x_2^n, x_2^{n^3}+x_1^n))$.  Notice that $x_1^2-x_2^n-( x_2^{n^3}+x_1^n) = x_2^n(-1-x_2^{n^3-n})$, and so $x_1^n, x_2^n \in I_n$.  Hence, \[
\ell(R/I_n) = \ell(R/(x_1, x_2^{n^3}))+\ell(R/(x_1^n-x_2^n, x_2^{n^3}+x_1^n)) = n^3+n^2.
\]  Of course, $\displaystyle \lim_{n \rightarrow \infty} \dfrac{n^3}{n^3+n^2} = 1\neq 0$.  
\end{proof}

We give the corresponding proposition over a Gorenstein ring of the desired type.

\begin{proposition} \label{d=2mixed} 
Let $R = V[[x_1, \ldots, x_{d}]]/(p^sx_1)$ where $V = (V, pV, \kappa)$ is a discrete valuation ring and $s \geq 1$.  Then neither $N = R/(p^k, x_d)$ for any $k \geq 1$ nor $N = R/(x_1, x_d)$ is $(d-1)$-effaceable for $d \geq 2$.
\end{proposition} 

\begin{proof}
As in the proceeding proposition, we let $a = p^k$ if $N = V[[x_1, \ldots, x_{d}]]/(p^k, x_d)$ and $a = x_1$ if $N = V[[x_1, \ldots, x_{d}]]/(x_1, x_d)$.  Again, we note that it is sufficient to find a sequence of parameter ideals $I_n \subseteq m^n$ such that $\displaystyle \lim_{n \to\infty} \dfrac{\ell(R/((a)+I_n))}{\ell(R/I_n)} \neq 0$ and that we need only consider $d=2$.  Because $\ell(R/((p)+I^n)) \leq \ell(R/((p^k)+I^n)) \leq k \cdot \ell(R/((p)+I^n))$ for each $k \geq 1$, it is sufficient to consider the case $k=1$.  

We will take $I_n = (x_2^{n+1}-x_2(p-x_1)^n, (p-x_1)^{n^3}+x_2^n)$.  It is clear that $\ell\left(\dfrac{V[[x_1, x_2]]}{I_n+(a, x_2)}\right)= n^3$.  Meanwhile, $\min(x_2^{n+1}-x_2(p-x_1)^n, (p-x_1)^{n^3}+y^n) = \{(x_2, p-x_1)\}$ for all $n \geq 1$ and so neither $p$ nor $x_1$ is a zerodivisor on $\dfrac{V[[x_1, x_2]]}{(x_2^{n+1}-x_2(p-x_1)^n, (p-x_1)^{n^3}+y^n)}$.  Therefore, \[
\ell\left(\dfrac{V[[x_1, x_2]]}{I_n+(p^sx_1)}\right)= \ell\left(\dfrac{V[[x_1, x_2]]}{I_n+(p^s)}\right)+\ell\left(\dfrac{V[[x_1, x_2]]}{I_n+(x_1)}\right) \leq s \cdot \ell\left(\dfrac{V[[x_1, x_2]]}{I_n+(p)}\right)+\ell\left(\dfrac{V[[x_1, x_2]]}{I_n+(x_1)}\right).
\]  The computation of $\ell\left(\dfrac{V[[x_1, x_2]]}{I_n+(p)}\right)$ and $\ell\left(\dfrac{V[[x_1, x_2]]}{I_n+(x_1)}\right)$ are symmetric and both equal to $n^3+n^2$, following the computation in Proposition \ref{d=2}.  Therefore, \[
\displaystyle \lim_{n \to\infty} \dfrac{\ell(R/((a, x_2)+I_n))}{\ell(R/I_n)} \geq \dfrac{n^3}{(s+1)(n^3+n^2)}  = \frac{1}{s+1} \neq 0,
\] completing the proof.

\end{proof}

%

\end{document}